\documentclass[a4paper]{article}

\usepackage[T1]{fontenc}
\usepackage[utf8]{inputenc}
\usepackage[francais]{layout}
\usepackage{csquotes}
\usepackage[french,english]{babel}

\usepackage{amsmath,amssymb,amsfonts,mathabx,stmaryrd}
\usepackage{amsthm}
\usepackage{xargs}
\usepackage{mdwlist}
\usepackage{tikz-cd}
\usepackage{ifthen}

\usepackage{soul}
\setuldepth{a}

\usepackage{graphicx}
\usepackage{enumitem}

\usepackage{ltugcomn}

\usepackage{ifdraft}
\usepackage{comment}

\usepackage[final]{hyperref}

\makeatletter
\hypersetup{
    unicode=true,
    pdftoolbar=true,
    pdfmenubar=true,
    pdffitwindow=false,
    pdfstartview={FitH},
    pdftitle={\@title},
    pdfauthor={\@author},
    pdfsubject={},
    pdfcreator={},
    pdfproducer={},
    pdfkeywords={},
    pdfnewwindow=true,
    colorlinks,
    linkcolor=black,
    citecolor=black,
    filecolor=black,
    urlcolor=black
}
\makeatother

\newcommand{\thmlevel}{subsection}
\usepackage{aliascnt}
\def\NewTheorem#1{%
  \newaliascnt{#1}{thm}
  \newtheorem{#1}[#1]{\csname #1Name\endcsname}
  \aliascntresetthe{#1}
  \expandafter\def\csname #1autorefname\endcsname{\csname #1Name\endcsname}
  \expandafter\def\csname #1Autorefname\endcsname{\csname #1Name\endcsname}
}

\newcommand{\theoremName}{\iflanguage{francais}{Th\'eor\`eme}{Theorem}}

\newcommand{\pbName}{\iflanguage{francais}{Probl\`eme}{Problem}}

\newcommand{\dfName}{\iflanguage{francais}{D\'efinition}{Definition}}

\ifthenelse{\equal{\thmlevel}{}}{\newtheorem{thm}{\theoremName}}{\newtheorem{thm}{\theoremName}[\thmlevel]}

\newtheorem*{thm*}{\theoremName}

\newtheorem{thmintro}{\theoremName}

\NewTheorem{lem}
\NewTheorem{sublem}
\NewTheorem{prop}
\NewTheorem{cor}
\NewTheorem{conj}
\NewTheorem{pb}

\theoremstyle{definition}
\NewTheorem{df}
\NewTheorem{exo}
\NewTheorem{ex}
\NewTheorem{rmq}

\newtheorem*{df*}{Definition}

\makeatletter
\renewenvironment{proof}[1][]{\par
  \pushQED{\qed}%
  \normalfont \topsep6\p@\@plus6\p@\relax
  \trivlist
  \item[\hskip\labelsep
        \bfseries
    \proofname\ifthenelse{\equal{#1}{}}{}{\textmd{ (#1)}}\@addpunct{.}]\ignorespaces
}{%
  \popQED\endtrivlist\@endpefalse
}
\makeatother

\makeatletter
\@addtoreset{thm}{section}
\makeatother
\newcommand{\N}{\mathbb{N}}
\newcommand{\Z}{\mathbb{Z}}
\newcommand{\Q}{\mathbb{Q}}
\newcommand{\R}{\mathbb{R}}
\newcommand{\Lcot}{\mathbb{L}}
\newcommand{\T}{\mathbb{T}}
\newcommand{\id}{\operatorname{id}}
\newcommand{\dual}[1]{{#1}^{\vee}}
\newcommand{\op}{^{\mathrm{op}}}
\newcommand{\noloc}{\,:}
\newcommand{\Cc}{\mathcal{C}}
\newcommand{\Dd}{\mathcal{D}}
\newcommand{\Oo}{\mathcal{O}}
\newcommand{\sSets}{\mathbf{sSets}}
\newcommand{\Sp}{\mathbf{Sp}}
\newcommand{\cdgaunbounded}{\mathbf{cdga}}
\newcommand{\cdga}{\cdgaunbounded^{\leq 0}}
\newcommand{\Alg}{\mathbf{Alg}}
\newcommand{\dgArt}{\mathbf{dgArt}}
\newcommand{\dgLie}{\mathbf{dgLie}}
\newcommand{\dgAlg}{\mathbf{dgAlg}}
\newcommand{\dgMod}{\mathbf{dgMod}}

\newcommand{\Perf}{\mathbf{Perf}}
\newcommand{\Map}{\operatorname{Map}}

\newcommand{\Fct}{\operatorname{Fct}}
\DeclareMathOperator*{\colim}{colim}
\newcommand{\BGL}{\operatorname{BGL}}
\newcommand{\B}{\mathrm{B}}
\newcommand{\E}{\operatorname{E}}
\newcommand{\K}{\operatorname{K}}
\newcommand{\gl}{\mathfrak{gl}}
\newcommand{\Sym}{\operatorname{Sym}}
\newcommand{\MCdgMod}{\mathrm{dgMod}}
\newcommand{\MCcdga}{\mathrm{cdga}^{\leq 0}}
\newcommand{\MCdgAlg}{\mathrm{dgAlg}}
\newcommand{\homol}{\mathrm H}
\newcommand{\End}{\operatorname{End}}
\newcommand{\Spec}{\operatorname{Spec}}
\newcommand{\GL}{\mathbb{GL}}
\newcommand{\Gm}{\mathbb{G}_m}
\newcommand{\eldebutpardefaut}{1}
\newcommandx*{\el}[4][2=\eldebutpardefaut,4={,}]{#1_{#2}#4\dots#4#1_{#3}}
\newcommand{\quot}[2]{\ensuremath \mathchoice {\displaystyle #1 \raisebox{-2pt}{$\displaystyle \hspace{-1pt}{/} $} \raisebox{-4pt}{$\displaystyle \hspace{-1pt}{#2}$}}{\textstyle #1 \raisebox{-1pt}{$\textstyle \hspace{-1pt}{/} $} \raisebox{-2pt}{$\textstyle \hspace{-1pt}{#2}$}}{\scriptstyle #1 \raisebox{-1pt}{$\scriptstyle \hspace{-1pt}{/} $} \raisebox{-2pt}{$\scriptstyle \hspace{-1pt}{#2}$}}{\scriptscriptstyle #1 \raisebox{-1pt}{$\scriptscriptstyle \hspace{-1pt}{/} $} \raisebox{-2pt}{$\scriptscriptstyle \hspace{-1pt}{#2}$}}}
\newcommand{\cart}{\arrow[dr, phantom,
	"\text{\tikz[baseline=0em,xscale=0.04em,yscale=0.04em]
    \draw[-] (-1,0) -- (0,0) -- (0,1) ;}\mkern1mu\relax",
    very near start]}
\let\originalleft\left
\let\originalright\right
\renewcommand{\left}{\mathopen{}\mathclose\bgroup\originalleft}
\renewcommand{\right}{\aftergroup\egroup\originalright}

\makeatletter
\def\DeclareMathBinOp{\@ifstar{\declaremathbinop@star}{\declaremathbinop@nostar}}
\def\declaremathbinop@star#1#2{\def#1{\test@subnexp@star#2}}
\def\test@subnexp@star#1{\@ifnextchar_{\isol@subnexp@star#1}{\test@exp@star#1}}
\def\test@exp@star#1{\@ifnextchar^{\isol@expnsub@star#1}{\mathbin{#1}}}
\def\isol@subnexp@star#1_#2{\@ifnextchar^{\eval@subnexp@star#1_#2}{\mathbin{\operatorname*{#1}_{#2}}}}
\def\eval@subnexp@star#1_#2^#3{\mathbin{\operatorname*{#1}_{#2}^{#3}}}
\def\isol@expnsub@star#1^#2{\@ifnextchar_{\eval@expnsub@star#1^#2}{\mathbin{\operatorname*{#1}^{#2}}}}
\def\eval@expnsub@star#1^#2_#3{\mathbin{\operatorname*{#1}_{#3}^{#2}}}
\def\declaremathbinop@nostar#1#2{\def#1{\test@subnexp@nostar#2}}
\def\test@subnexp@nostar#1{\@ifnextchar_{\isol@subnexp@nostar#1}{\test@exp@nostar#1}}
\def\test@exp@nostar#1{\@ifnextchar^{\isol@expnsub@nostar#1}{\mathbin{#1}}}
\def\isol@subnexp@nostar#1_#2{\@ifnextchar^{\eval@subnexp@nostar#1_#2}{\mathbin{\underset{#2}{#1}}}}
\def\eval@subnexp@nostar#1_#2^#3{\mathbin{\overset{#3}{\underset{#2}{#1}}}}
\def\isol@expnsub@nostar#1^#2{\@ifnextchar_{\eval@expnsub@nostar#1^#2}{\mathbin{\overset{#2}{#1}}}}
\def\eval@expnsub@nostar#1^#2_#3{\mathbin{\overset{#2}{\underset{#3}{#1}}}}
\makeatother

\let\OLDtimes\times
\DeclareMathBinOp*{\times}{\OLDtimes}
\let\OLDamalg\amalg
\DeclareMathBinOp*{\amalg}{\OLDamalg}
\let\OLDotimes\otimes
\DeclareMathBinOp*{\otimes}{\OLDotimes}
\let\OLDwedge\wedge
\DeclareMathBinOp*{\wedge}{\OLDwedge}

\makeatletter
\def\DeclareArrow#1#2{\def#1{\test@subnexp#2}}
\def\test@subnexp#1{\@ifnextchar_{\isol@subnexp#1}{\test@exp#1}}
\def\test@exp#1{\@ifnextchar^{\isol@expnsub#1}{#1}}
\def\isol@subnexp#1_#2{\@ifnextchar^{\eval@subnexp#1_#2}{\underset{#2}{#1}}}
\def\eval@subnexp#1_#2^#3{\underset{#2}{\overset{#3}{#1}}}
\def\isol@expnsub#1^#2{\@ifnextchar_{\eval@expnsub#1^#2}{\overset{#2}{#1}}}
\def\eval@expnsub#1^#2_#3{\overset{#2}{\underset{#3}{#1}}}
\makeatother

\let\OLDto\to
\DeclareArrow{\to}{\OLDto}
\DeclareArrow{\from}{\leftarrow}
\DeclareArrow{\lra}{\longrightarrow}
\DeclareArrow{\lla}{\longleftarrow}

\newcommand\OLDbarotimes{\bar{\otimes}}
\DeclareMathBinOp*{\barotimes}{\OLDbarotimes}

\usepackage{geometry}
\setlist[enumerate]{label=\emph{(\roman*)},ref=\emph{(\roman*)}}

\newlist{assertions}{enumerate}{1}
\setlist[assertions]{label={\rm(\alph*)},  ref={assertion (\alph*)}}

\newlist{conditions}{enumerate}{1}
\setlist[conditions]{label={(\arabic*)}, ref={condition (\arabic*)}}
\newlist{assumptions}{enumerate}{1}
\setlist[assumptions]{label={(\roman*)}, ref={assumption (\roman*)}}
\newlist{disjunction}{enumerate}{1}
\setlist[disjunction]{label={(\arabic*)}, ref={case (\arabic*)}}

\newcommand{\Bc}{\mathcal{B}}

\newcommand{\Lc}{\mathcal{L}}
\renewcommand{\Mc}{\mathcal{M}}
\newcommand{\Qc}{\mathcal{Q}}

\newcommand{\Hc}{\mathcal{H}}
\newcommand{\Fc}{\mathcal{F}}

\newcommand{\bch}{\widebar{\operatorname{ch}}}

\renewcommand{\GL}{\operatorname{GL}}
\newcommand{\bK}{\widebar{\mathrm{K}}}
\newcommand{\bBGL}{\widebar{\BGL}}
\newcommand{\FMP}{\mathbf{FMP}}
\newcommand{\HC}{\mathrm{HC}}
\newcommand{\HH}{\mathrm{HH}}
\newcommand{\bHH}{\widebar{\mathrm{HH}}{}}
\newcommand{\Aug}{\operatorname{Aug}}
\newcommand{\bHC}{\widebar{\mathrm{HC}}{}}
\newcommand{\infloop}{\operatorname{\Omega^\infty}}
\newcommand{\infsusp}{\operatorname{\Sigma^\infty}}
\newcommand{\FAb}{\operatorname{F^{Ab}}}
\newcommand{\GAb}{\operatorname{G_{Ab}}}
\newcommand{\Ab}{\mathbf{Ab}}

\renewcommand{\k}{\mathbf{k}}

\newcommand{\e}{\operatorname{e}}
\newcommand{\eAb}{\operatorname{e_{Ab}}}
\newcommand{\eQ}{\operatorname{e_\Q}}

\newcommand{\incl}{\operatorname{i}}
\newcommand{\iAb}{\operatorname{i_{Ab}}}
\newcommand{\iQ}{\operatorname{i_{\Q}}}

\newcommand{\iC}{\operatorname{i_\Cc}}

\newcommand{\jQ}{\operatorname{j_\Q}}

\newcommand{\ellAb}{\operatorname{\ell^{Ab}}}
\newcommand{\ellQ}{\operatorname{\ell^\Q}}

\newcommand{\SQ}{\operatorname{S^\Q}}
\newcommand{\SQb}{\operatorname{{\bar S}^\Q}}

\newcommand{\I}{\operatorname{I}}
\newcommand{\Ib}{\operatorname{{\bar I}}}
\newcommand{\Ac}{\mathcal{A}}
\newcommand{\bAc}{\widebar{\mathcal{A}}}

\newcommand{\TAb}{\operatorname{\T^{Ab}}}
\newcommand{\TQ}{\operatorname{\T^\Q}}

\newcommand{\form}{\operatorname{L}}
\newcommand{\LAb}{\operatorname{L^{Ab}}}
\newcommand{\LC}{\operatorname{L^\Cc}}
\newcommand{\LQ}{\operatorname{L^{\Q}}}

\newcommand{\bCE}{\widebar{\mathrm{CE}}}
\newcommand{\bCEO}{\widebar{\mathrm{CE}}{}^\Omega_\bullet}
\newcommand{\CE}{\mathrm{CE}}
\newcommand{\CEO}{\mathrm{CE}{}^\Omega_\bullet}

\newcommand{\PFMP}{\mathbf{PFMP}}
\newcommand{\hofib}{\operatorname{hofib}}

\newcommand{\MCdgBiMod}{\mathrm{dgBiMod}}
\newcommand{\dgBiMod}{\mathbf{dgBiMod}}
\newcommand{\sAlg}{\mathrm{sAlg}}
\newcommand{\Tr}{\operatorname{Tr}}
\newcommand{\n}{\mathfrak n}
\newcommand{\g}{\mathfrak{g}}
\newcommand{\bBG}{\widebar{\mathrm{BG}}}

\newcommand{\PicZ}{\operatorname{Pic}^\Z}

\title{The tangent complex of K-theory}
\author{Benjamin Hennion\thanks{IMO - Université Paris-Saclay - benjamin.hennion@universite-paris-saclay.fr}}
\date{\today}

\begin{document}

\maketitle
\begin{abstract}
We prove that the tangent complex of K-theory, in terms of (abelian) deformation problems over a characteristic $0$ field $\k$, is cyclic homology (over $\k$).
This equivalence is compatible with the $\lambda$-operations.
In particular, the relative algebraic K-theory functor fully determines the absolute cyclic homology over any field $\k$ of characteristic $0$.

We also show that the Loday-Quillen-Tsygan generalized trace comes as the tangent morphism of the canonical map $\BGL_\infty \to \K$.

The proof builds on results of Goodwillie, using Wodzicki's excision for cyclic homology and formal deformation theory à la Lurie-Pridham.
\end{abstract}
\tableofcontents

\addcontentsline{toc}{section}{Introduction}%

\section*{Introduction}

Computing the tangent space of algebraic $\K$-theory has been the subject of many articles. The first attempt known to the author is due to Spencer Bloch \cite{bloch:tangentk} in 1973. It was then followed by a celebrated article of Goodwillie \cite{goodwillie:tangentk} in 1986.

Considering the following example, we can easily forge an intuition on the matter. Let $A$ be a smooth commutative $\Q$-algebra and $G(A)$ be the group $G(A) = [\GL_\infty(A),\GL_\infty(A)]$ of elementary matrices. It admits a universal central extension
\[
 1 \lra \K_2(A) \lra G(A)^+ \lra G(A) \lra 1
\]
by the second $\K$-theory group of $A$. The group $G(A)^+$ is the Steinberg group of $A$. This above exact sequence can also be extended as an exact sequence
\begin{equation}\label{eq:steinberg}
 1 \lra \K_2(A) \lra G(A)^+ \lra \GL_\infty(A) \lra \K_1(A) \lra 1.
\end{equation}
This exact sequence can be thought as the universal extension of $\GL_\infty$ by $\K$-theory (both $\K_2$ and $\K_1$ here). This idea leads to Quillen's definition of $\K$-theory through the $+$-construction.

Consider now the tangent Lie algebra $\gl_\infty$ of $\GL_\infty \colon B \mapsto \GL_\infty(B)$. Its current Lie algebra $\gl_\infty(A) := \gl_\infty \otimes_\Q A$ also admits a universal central extension, this time by the first cyclic homology group
\begin{equation} \label{eq:LQTextension}
 0 \lra \HC_1^\Q(A) \lra \gl_\infty(A)^+ \lra \gl_\infty(A) \lra 0.
\end{equation}
The obvious parallel between those two central extensions leads to the idea that the (suitably considered) tangent space (or complex rather) of $\K$-theory should be cyclic homology. We will give meaning to this folkloric statement, prove it and provide a comparison between those extensions (see below).

In both of the aforementioned articles of Bloch and Goodwillie, the tangent space is considered in a rather naive sense: as if $\K$-theory were an algebraic group. Bloch defines the tangent space of $\K$-theory at $0$ in $\K(A)$ (for $A$ a smooth commutative algebra over $\Q$) as the fiber of the augmentation
\[
 \K(A[\varepsilon]) \to \K(A),
\]
where $\varepsilon$ squares to $0$. Goodwillie then extends and completes the computation  by showing that relative (rational) $\K$-theory is isomorphic to relative cyclic homology, in the more general setting where $A$ is a simplicial associative $\Q$-algebra. He shows that for any nilpotent extension $A'$ of $A$, the homotopy fibers of $\K(A') \otimes \Q \to \K(A) \otimes \Q$ and of $\HC^\Q_\bullet(A') \to \HC^\Q_\bullet(A)$ are quasi-isomorphic.

In this article, we give another definition of the tangent space of $\K$-theory using deformation theory, over any field $\k$ containing $\Q$. We then show that this tangent space is equivalent to the \emph{absolute} and \emph{$\k$-linear} cyclic homology. Before explaining exactly how this tangent space is defined, let us state the main result. In this introduction, we will restrict for simplicity to the connective case\footnote{For the unbounded case, we essentially replace in what follows $A \otimes_\k B$ with its connective cover $(A \otimes_\k B)^{\leq 0}$.} (or equivalently to the case of simplicial algebras).
With our definition of tangent complex, for any unital simplicial $\k$-algebra $A$, we have
\[
 \T_{\K(A),0} \simeq \HC_{\bullet-1}^\k(A)
\]
where the right-hand-side denotes the (shifted) $\k$-linear (absolute) cyclic homology of $A$. Of course, for this to hold for any field $\k$, the left-hand side has to depend on $\k$.
This dependence occurs by only considering relative $\K$-theory of nilpotent extensions $A'$ of $A$ of the form
\[
 A' = A \otimes_\k B \to A
\]
where $B$ is a (dg-)Artinian commutative $\k$-algebra with residue field $\k$.
This defines a functor 
\begin{align*}
 \bK(\Ac_A) \colon \dgArt_\k &\to \Sp_{\geq 0} \\
 B \hspace{1.2em} & \textstyle \mapsto \hofib\left(\K\left(A \otimes_\k B\right) \to \K(A) \right)
\end{align*}
from the category of dg-Artinian commutative $\k$-algebras with residue field $\k$ to the category of connective spectra.
The category of such functors $\dgArt_\k \to \Sp_{\geq 0}$ admits a full subcategory of formal deformation problems -- i.e. of functors satisfying a Schlessinger condition (see \autoref{df:schlessinger}). The datum of such a functor is now equivalent to the datum of a complex of $\k$-vector spaces.
The induced fully faithful functor $\dgMod_\k \to \Fct(\dgArt_\k, \Sp_{\geq 0})$ admits a left adjoint -- denoted by $\ellAb$ -- that forces the Schlessinger conditions.

We define\footnote{This definition is somewhat close to this idea of Goodwillie derivatives, but for functors defined on categories of non-abelian nature (of commutative algebras in our case). Replacing $\dgArt_\k$ with complexes of $\k$-vector spaces would give us the Goodwillie derivative of the $\K$-theory functor, namely Hochschild homology. See \autoref{rmk:goodwilliederivative}.}
the tangent complex of $\K$-theory (of $A$) as
\[
 \T_{\K(A),0} := \ellAb(\bK(\Ac_A))
\]
and our main theorem now reads
\begin{thmintro}[see \autoref{cor:tgtK}]\label{thmintro}
 Let $A$ be any (H-)unital dg-algebra over $\k$, with $\mathrm{char}(\k) = 0$.
 There is a natural equivalence
 \[
  \homol^{-\bullet}\left(\T_{\K(A),0} \right) \simeq \HC_{\bullet-1}^\k(A).
 \]
 This equivalence is furthermore compatible with the $\lambda$-operations on each side.
\end{thmintro}
\noindent We later extend this result to the case where $A$ is replaced by a quasi-compact quasi-separated scheme.

As a consequence, the $\K$-theory functor fully determines the cyclic homology functors over all fields of characteristic $0$.
Moreover, it gives its full meaning to the use of ``Additive $\K$-theory'' as a name of cyclic homology (see \cite{feigintsygan:addK}).

Using our theorem, we then prove (see \autoref{thm:compLQT}) that the canonical natural transformation $\BGL_\infty \to \K$ (encoding the aforementioned universal central extension of $\GL_\infty$) induces a morphism $\gl_\infty(A) \to \HC_\bullet^\k(A)$ of homotopical Lie algebras (i.e. a $\mathcal L_\infty$-morphism). Such a morphism corresponds to a morphism of complexes
\[
 \CE_\bullet^{\k}(\gl_\infty(A)) \to \HC_{\bullet-1}^\k(A)
\]
from the Chevalley-Eilenberg homological complex to cyclic homology. We will show that this morphism identifies with the Loday-Quillen-Tsygan generalized trace. In particular, it exhibits $\HC_{\bullet - 1}^\k(A)$ as the kernel of the universal central extension of $\gl_\infty(A)$.
As a consequence, the extensions (\ref{eq:steinberg}) and (\ref{eq:LQTextension}) above are indeed tangent to one another.

\subsection*{Structure of the proof}
The proof of \autoref{thmintro} goes as follows. First, we show in \autoref{subsec:FMPQ} that the tangent complex $\T_{\K(A),0}$ only depends on the (relative) rational $\K$-theory functor $\bK \wedge \Q$.
Using the work of Goodwillie \cite{goodwillie:tangentk}, we know the relative rational 
$\K$-theory functor is equivalent (through the Chern character), with the relative rational cyclic homology functor: $\bK \wedge \Q \simeq \bHC_{\bullet-1}^\Q$.
We get 
\[
 \T_{\K(A),0} := \ellAb(\bK(\Ac_A)) \simeq \ellQ\left(\bHC_{\bullet -1}^\Q(\Ac_A)\right)
\]
where $\bHC^\Q_{\bullet-1}(\Ac_A) \colon \dgArt_\k \to \dgMod_\Q^{\leq 0}$ maps an Artinian dg-algebra $B$ over $\k$ to the connective $\Q$-dg-module $\hofib(\HC_{\bullet -1}^\Q(A \otimes_\k B) \to \HC_{\bullet -1}^\Q(A))$, and where $\ellQ$ is the left adjoint of the fully faithful functor $\dgMod_\k \to \Fct(\dgArt_\k,\dgMod_\Q^{\leq 0})$ mapping $V$ to $B \mapsto (V \otimes_\k \Aug(B)) ^{\leq 0}$.

Since cyclic homology if well defined for non-unital algebras, we also have a functor $\HC_{\bullet-1}^\Q(\bAc_A)$ mapping an Artinian $B$ to the (shifted) cyclic homology of the augmentation ideal $\bAc_A(B) := A \otimes_\k \Aug(B)$.

We then argue that the functor $\ellQ$ is (non-unitally) symmetric monoidal (once restricted to a full subcategory, see \autoref{prop:monoidal}). Since $\ellQ(A \otimes_\k \Aug(-)) = A$, this implies the equivalence $\ellQ \HC_{\bullet-1}^\Q(\bAc_A) \simeq \HC_{\bullet-1}^\k(A)$.

We will then prove (see \autoref{thm:excision}) that the induced morphism
\[
 \ellQ\HC_{\bullet-1}^\Q(\bAc_A) \to^\sim \ellQ\bHC_{\bullet-1}^\Q(\Ac_A)
\]
is a quasi-isomorphism.
We use here the assumption that $A$ is unital (or at least H-unital). This excision statement is close to Wodzicki's excision theorem for cyclic homology \cite{wodzicki:excision}. The structure of our proof relies on a paper by Guccione and Guccione \cite{guccioneguccione}, where the authors give an alternative proof of Wodzicki's theorem.
Nonetheless, our \autoref{thm:excision} is strictly speaking not a consequence of Wodzicki's theorem, and the proof is somewhat more subtle.

Composing those quasi-isomorphisms, we find the announced theorem
\[
 \T_{\K(A),0} := \ellAb(\bK(\Ac_A)) \simeq \ellQ(\bHC_{\bullet-1}^\Q(\Ac_A)) \simeq \ellQ(\HC_{\bullet -1}^\Q(\bAc_A)) \simeq \HC_{\bullet-1}^\k(A).
\]

\subsection*{Possible generalizations}
In this article, we work (mostly for simplicity) over a field $\k$ of characteristic $0$. The results will also hold over any commutative $\Q$-algebra, or, more generally, over any eventually coconnective simplicial $\Q$-algebra\footnote{so that the theory of formal moduli problems would still work flawlessly.}. This can surely also be done over more geometric bases like schemes or stacks (and bounded enough derived versions of such).

A more interesting generalization would be to (try to) work over the sphere spectrum. Since we only need in what follows 'abelian' formal moduli problems, it is not so clear that the characteristic $0$ is necessary. There would, however, be significant difficulties to be overcome, starting with a Wodzicki's excision theorem for topological cyclic homology.

\subsection*{Acknowledgements}
The question answered in this text naturally appeared while working on \cite{faontehennionkapranov:kacmoody} with G. Faonte and M. Kapranov.
A discussion with G. Ginot and M. Zeinalian raised the problem dealt with in our last section. I thank them for the many discussions we had, that led to this question.
I thank D. Calaque, P-G. Plamandon, J. Pridham and M. Robalo for useful discussions on the content of this article.
I of course thank the anonymous referees for their very constructive comments on the first drafts of this article.

Finally, I thank J. Pridham for bringing \cite{pridham:smoothfunctions} to my attention when the first version of our work appeared online. Some arguments used in \cite{pridham:smoothfunctions} are fairly similar to those we use here.

\subsection*{Notations}
From now on, we fix the following notations
\begin{itemize}
 \item Let $\k$ be a field of characteristic $0$.
 \item Let $\MCdgMod_\k$ denote the category of (cohomologically graded) complexes of $\k$-vector spaces. Let $\MCdgMod_\k^{\leq 0}$ be its full subcategory of connective objects (i.e. $V^\bullet$ such that $V^n = 0$ for $n > 0$).
 Let $\dgMod_\k$ and $\dgMod_\k^{\leq 0}$ denote the $\infty$-categories obtained from the above by inverting the quasi-isomorphisms.
 \item Let $\MCdgAlg_\k^{\mathrm{nu}}$ be the category of (possibly non-unital) associative algebras in $\MCdgMod_\k$ (with its usual graded tensor product). We denote by $\MCdgAlg_\k^{\mathrm{nu},\leq 0}$ its full subcategory of connected objects and by $\dgAlg^{\mathrm{nu},\leq 0}_\k \subset \dgAlg_\k^{\mathrm{nu}}$ the associated $\infty$-categories.
 \item For $A \in \MCdgAlg_\k^{\mathrm{nu},\leq 0}$, we denote by $\MCdgBiMod_A^{\mathrm{nu},\leq 0}$ the category of connective $\k$-complexes with a left and a right action of $A$. We denote by $\dgBiMod^{\mathrm{nu},\leq 0}_A$ the associated $\infty$-category\footnote{In the case where $A$ was unital to begin with, the $\infty$-category of connective $A$-bimodules embeds fully faithfully into $\dgBiMod^{\mathrm{nu},\leq 0}_A$. Its image is spanned by those modules on which a unit acts by an equivalence. See \cite[5.4.3.5 and 5.4.3.14]{lurie:halg}.}.
 \item Let $\MCcdga_\k$ denote the category of connective commutative dg-algebras over $\k$. We denote by $\cdga_\k$ its $\infty$-category.
 \item Let $\sSets$ be the $\infty$-category of spaces, $\Sp_{\geq 0}$ the $\infty$-category of connective spectra, and $\infsusp \colon \sSets \rightleftarrows \Sp_{\geq 0} \noloc \infloop$ the adjunction between the infinite suspension and loop space functors.
\end{itemize}

\section{Relative cyclic homology and \texorpdfstring{$\K$}{K}-theory}

In this first section, we will introduce cyclic homology, $\K$-theory and the relative Chern character between them. Most of the content has already appeared in the literature.
The only original fragment is the extension of some of the statements and proofs to simplicial H-unital algebras.

\subsection{Hochschild and cyclic homologies}\label{subsec:cyclic}

\paragraph{Definitions}
Fix an associative dg-algebra $A \in \dgAlg^\mathrm{nu}_\k$. Assuming (for a moment) that $A$ is unital, its Hochschild homology is 
\[
 \HH_\bullet^\k(A) = A \otimes^{\Lcot}_{A \otimes_\k A^o} A.
\]
It comes with a natural action of the circle, and we define its cyclic homology $\HC_\bullet^\k(A)$ to be the (homotopy) coinvariants $\HH_\bullet^\k(A)_{hS^1}$ under this action.
To define those homologies for a non-unital algebra $A$, we first formally add a unit to $A$ and form $A^+ \simeq A \oplus \k$. We then define $F(A) = \operatorname{hocofib}(F(\k) \to F(A^+))$ for $F$ being either $\HH_\bullet^\k$ or $\HC_\bullet^\k$. Those definitions turn out to agree with the former ones when $A$ was already unital.

Unfortunately, we will need later down the road a construction of $\HC_\bullet^\k(A)$ for $A$ non-unital that does not rely on the one for unital algebras. We will therefore work with the following explicit models. We will first define strict functors, and then invert the quasi-isomorphisms.

We fix $A \in \MCdgAlg_\k^{\mathrm{nu}}$ a (not necessarily unital) associative algebra in complexes over $\k$. We also fix $M$ an $A$-bimodule.
Throughout this section, the tensor product $\otimes$ will always refer to the tensor product over $\k$.

 \begin{df}
We call the (augmented) Bar complex of $A$ with coefficient in $M$ and denote by $\Bc_\bullet^\k(A,M)$ the $\oplus$-total complex of the bicomplex
\[
\begin{tikzcd}
   \cdots \arrow[r] & \displaystyle M \otimes_\k A^{\otimes 2} \arrow[r, "-b'"] & \displaystyle M \otimes_\k A \arrow[r,"-b'"] & M \arrow[r] & 0
\end{tikzcd}
\]

with $M$ in degree $0$ and with differential $-b' \colon M \otimes A^{\otimes n} \to M \otimes A^{\otimes n-1}$ given on homogeneous elements by
\[
 b'(a_0 \otimes \cdots \otimes a_n) = \sum_{i = 0}^{n-1} (-1)^{\epsilon_i} a_0 \otimes \cdots \otimes a_i a_{i+1} \otimes \cdots \otimes a_n
\]
where $\epsilon_i = i + \sum_{j<i}|a_j|$ ($|a|$ standing for the degree of the homogeneous element $a$). One easily checks that $b'$ squares to $0$ and commutes with the internal differentials of $A$ and $M$.

We denote by $\Bc_\bullet^\k(A)$ the complex $\Bc_\bullet^\k(A,A)$.
\end{df}

 Assuming $A$ is unital and the right action of $A$ on $M$ is unital, we can build a nullhomotopy of $\Bc_\bullet^\k(A,M) \simeq 0$.
This contractibility does not hold for general non-unital algebras and modules.
\begin{df}[Wodzicki]
 The dg-algebra $A$ as above is called H-unital if $\Bc_\bullet^\k(A) \simeq 0$.
 The $A$-bimodule $M$ is called H-unitary over $A$ if $\Bc_\bullet^\k(A,M) \simeq 0$.
\end{df}

\begin{rmq}
 In the above definitions, we only used the right action of $A$ on $M$. In the original article of Wodzicki, such a module $M$ would be called right H-unitary. A similar notion exists for left modules.
\end{rmq}

\begin{rmq}\label{rmq:MtimesAHunital}
 If $A$ is H-unital, then its right module $M \otimes A$ is H-unitary, for any $M$. Indeed, we then have $\Bc_\bullet^\k(A,M \otimes A) \simeq M \otimes \Bc_\bullet^\k(A) \simeq 0$.
\end{rmq}

\begin{df}
 We denote by $\Hc_\bullet^\k(A,M)$ the $\oplus$-total complex of the bicomplex
 \[
\begin{tikzcd}
   \cdots \arrow[r] & \displaystyle M \otimes_\k A^{\otimes 2} \arrow[r, "b"] & \displaystyle M \otimes_\k A \arrow[r, "b"] & M \arrow[r] & 0
\end{tikzcd}
 \]
with $M$ in degree $0$ and with differential $
 b \colon M \otimes A^{\otimes n} \to M \otimes A^{\otimes n-1}$ given on homogeneous elements by
\[
 b(a_0 \otimes \cdots \otimes a_n) = b'(a_0 \otimes \cdots \otimes a_n) + (-1)^{(|a_n|+1)\epsilon_n} a_na_0 \otimes \cdots \otimes a_{n-1}.
\]
Here also, the differential squares to $0$ and is compatible with the internal differentials of $A$ and $M$.

We denote by $\Hc_\bullet^\k(A)$ the complex $\Hc_\bullet^\k(A,A)$.
\end{df}

 If $A$ is (H-)unital, the complex $\Hc_\bullet^\k(A)$ is the usual Hochschild complex, that computes the Hochschild homology of $A$. For general $A$'s, we need to compensate for the lack of contractibility of the Bar-complex with some extra-term.

 Let $t,N \colon A^{\otimes n+1} \to A^{\otimes n+1}$ by the morphisms given on homogeneous elements by the formulae
  \[
   t(a_0 \otimes \cdots \otimes a_n) = (-1)^{(|a_n|+1)\epsilon_n} a_n \otimes a_0 \otimes \cdots \otimes a_{n-1} \hspace{7mm} \text{and} \hspace{7mm} N = \sum_{i=0}^n t^i.
  \]
One easily checks that they define morphisms of complexes $1-t \colon \Bc_\bullet^\k(A) \to \Hc_\bullet^\k(A)$ and $N \colon \Hc_\bullet^\k(A) \to \Bc_\bullet^\k(A)$. Moreover we have $N(1-t) = 0$ and $(1-t)N = 0$. We can therefore define Hochschild and cyclic homology as follows.

\begin{df}
We define the Hochschild homology $\HH_\bullet^\k(A)$ and the cyclic homology $\HC_\bullet^\k(A)$ of $A$ as the $\oplus$-total complexes of the following bicomplexes
\[
 \begin{tikzcd}[row sep=tiny, column sep=1.9em]
\HH_\bullet^\k(A)\colon ~~ \cdots \arrow[r] & 0 \arrow[r] & 0 \arrow[r] & \Bc_\bullet^\k(A) \arrow[r, "1-t"] & \Hc_\bullet^\k(A) \arrow[r] & 0 \arrow[r] & \cdots 
\\
 \HC_\bullet^\k(A) \colon ~~ \cdots \arrow[r, "N"] & \Bc_\bullet^\k(A) \arrow[r, "1-t"] & \Hc_\bullet^\k(A) \arrow[r, "N"] & \Bc_\bullet^\k(A) \arrow[r, "1-t"] & \Hc_\bullet^\k(A) \arrow[r] & 0 \arrow[r] & \cdots
 \end{tikzcd}
\]
where in both cases, the rightmost $\Hc_\bullet^\k(A)$ is in degree $0$. 
 The Connes exact sequence is the obvious fiber and cofiber sequence
 \[
\begin{tikzcd}
 \HH_\bullet^\k(A) \arrow[r] & \HC_\bullet^\k(A) \arrow[r, "B"] & \HC_\bullet^\k(A)[2]
  \end{tikzcd}
 \]
induced by the above definitions.
\end{df}

\begin{rmq}
 The Hochschild homology of $A$ is the homotopy cofiber of $1-t \colon \Bc_\bullet^\k(A) \to \Hc_\bullet^\k(A)$. In particular, we have a fiber (and cofiber) sequence
 \[
 \begin{tikzcd}
 \Bc_\bullet^\k(A) \arrow[r, "1-t"] & \Hc_\bullet^\k(A) \arrow[r] & \HH_\bullet^\k(A).
 \end{tikzcd}
 \]
 If $A$ is H-unital, then we have $\Hc_\bullet^\k(A) \simeq \HH_\bullet^\k(A)$.
 Moreover, under this assumption, using the reduced Bar complex as a resolution of $A$ as a $A \otimes A^o$-dg-module, we easily show:
 \[
  \HH_\bullet^\k(A) \simeq A \otimes^\Lcot_{A \otimes A^o} A.
 \]
\end{rmq}

\begin{rmq}\label{rmq:HChS}
 The normalization $N$ can be seen as a morphism $N \colon \HH_\bullet^\k(A) \to \HH_\bullet^\k(A)[-1]$, which in turn is an action of (the $\k$-valued homology of) the circle $S^1$ on $\HH_\bullet^\k(A)$.
 Choosing a suitable resolution of $\k$ as an $\homol_\bullet(S^1) := \homol_\bullet(S^1,\k)$-module, we find
 \[
\HC_\bullet^\k(A) = {\HH_\bullet^\k(A)}_{hS^1} = \HH_\bullet^\k(A) \otimes^\Lcot_{\homol_\bullet(S^1)} \k.
 \]
\end{rmq}

\paragraph{Relation to Chevalley-Eilenberg homology}\label{par:tracemap}

Cyclic homology is sometimes referred to as additive $\K$-theory for the following reason: it is related to (the homology of) the Lie algebra $\gl_\infty(A)$ of finite matrices the same way $\K$-theory is related to (the homology of) the group $\GL_\infty(A)$.

More specifically, for $A$ a dg-algebra, the generalized trace map is a morphism
\[
 \Tr \colon \CE^\k_\bullet(\gl_\infty(A)) \to \HC_\bullet^\k(A) [1]
\]
used by Loday-Quillen \cite{lodayquillen} and Tsygan \cite{tsygan:HC} to prove (independently) the following statement for $A$ a discrete algebra, and by Burghelea \cite{burghelea:cyclicK} for general dg-algebras.
\begin{thm}[Loday-Quillen, Tsygan, Burghelea]\label{thm:LQT}
 When $A$ is unital, the morphism $\Tr$ induces an equivalence of Hopf algebras
 \[
  \CE^\k_\bullet(\gl_\infty(A)) \simeq \Sym_\k(\HC_\bullet^\k(A)[1]),
 \]
 where the product on the left-hand-side is given by the direct sum of matrices.
\end{thm}

\subsection{Filtrations, relative homologies and Wodzicki's excision theorem}\label{subsec:wodzicki}

\begin{df}
 Let $f \colon A \to B$ be a map of (possibly non-unital) connective\footnote{For simplicity, we restrict ourselves to the connective case. The general case -- unneeded for our purposes -- would work similarly.} dg-algebras.
 The relative Hochschild (resp. cyclic) homology of $A$ over $B$ is the homotopy fiber
 \begin{align*}
  \bHH_\bullet^\k(f) &:= \hofib( \HH_\bullet^\k(A) \to \HH_\bullet^\k(B)) \hspace{1.5em} \\ \text{\big(resp. } \bHC_\bullet^\k(f) &:= \hofib(\HC_\bullet^\k(A) \to \HC_\bullet^\k(B)) \text{ \big)}.
 \end{align*}
 If $I$ denotes the homotopy fiber of $f$ endowed with its induced (non-unital) algebra structure, we have canonical morphisms
 \[
  \eta_\HH \colon \HH_\bullet^\k(I) \to \bHH_\bullet^\k(f) \hspace{1em}\text{ and } \hspace{1em} \eta_\HC \colon \HC_\bullet^\k(I) \to \bHC_\bullet^\k(f).
 \]

\end{df}

\begin{thm}[Wodzicki]\label{thm:wodzicki}
 Let $I \to A \to B$ be an extension of (possibly non-unital) connective dg-algebras. 
 If $I$ is H-unital, then the induced sequences
 \[
 \begin{tikzcd}[row sep=tiny]
   \HH_\bullet^\k(I) \arrow[r] & \HH_\bullet^\k(A) \arrow[r] & \HH_\bullet^\k(B) \\
   \HC_\bullet^\k(I) \arrow[r] & \HC_\bullet^\k(A) \arrow[r] & \HC_\bullet^\k(B)
 \end{tikzcd}
 \]
 are fiber and cofiber sequences. In other words, the canonical morphisms $\eta_\HH$ and $\eta_\HC$ are equivalences.
\end{thm}

Wodzicki proves in \cite{wodzicki:excision} the above theorem in the case where $A$, $B$ and $I$ are concentrated in degree $0$. If his proof should be generalizable to connective dg-algebras, some computations seem to become tedious. Fortunately, Guccione and Guccione published in \cite{guccioneguccione} another proof of this result, that is very easily generalizable to connective dg-algebras. There is also a more recent article of Donadze and Ladra \cite{DonadzeLadra} proving this result for simplicial algebras.

We will actually not need \autoref{thm:wodzicki} in what follows. We will, however, need to reproduce some steps of its proof in a more complicated situation. 
In this subsection, we will give a short proof of \autoref{thm:wodzicki}, where we have isolated the statements to be used later. 
The proof follows closely the work of \cite{guccioneguccione}.

We fix $f \colon A \to B$ a degree-wise surjective morphism in $\MCdgAlg_\k^{\mathrm{nu},\leq 0}$. We denote by $I$ the kernel of $f$. Let $M \in \MCdgBiMod^{\mathrm{nu},\leq 0}_A$.

\paragraph{A filtration from $\Bc_\bullet^\k(I,M)$ to $\Bc_\bullet^\k(A,M)$ and from $\Hc_\bullet^\k(I,M)$ to $\Hc_\bullet^\k(A,M)$}
We introduce a filtration on both $\Bc_\bullet^\k(A,M)$ and $\Hc_\bullet^\k(A,M)$. Those filtrations were originally found in \cite{guccioneguccione}.
\label{par:filtrationF}

Fix $n \in \N$. Let $p \in \N$. We set 
\[
  M^n_p = \begin{cases}
  M \otimes A^p & \text{if } p \leq n \\
  M \otimes A^{\otimes n} \otimes I^{\otimes p-n} & \text{if } p > n.
  \end{cases}
\]
We denote by $\Fc_{\Bc_\bullet^\k}^n(f,M)$ (resp. $\Fc_{\Hc_\bullet^\k}^n(f,M)$) the subcomplex of $\Bc_\bullet^\k(A,M)$ (resp. $\Hc_\bullet^\k(A,M)$) given as the total complex of the bicomplex
\begin{align*}
  \Fc_{\Bc_\bullet^\k}^n(f,M) \colon& \hspace{1em} \cdots \lra M^n_p \lra^{-b'} M^n_{p-1} \lra^{-b'} \cdots \lra^{-b'} M^n_1 \lra^{-b'} M^n_0 \lra 0 \\
  \text{\big(resp. }\Fc_{\Hc_\bullet^\k}^n(f,M) \colon& \hspace{1em} \cdots \lra M^n_p \lra^{b} M^n_{p-1} \lra^{b} \cdots \lra^{b} M^n_1 \lra^{b} M^n_0 \lra 0 \text{ \big). }
\end{align*}
In particular, we have 
\begin{align*}
\Fc_{\Bc_\bullet^\k}^0(f,M) &= \Bc_\bullet^\k(I,M) \text{, }& \Bc_\bullet^\k(A,M) &\simeq \colim_n \Fc_{\Bc_\bullet^\k}^n(f,M) \\
\Fc_{\Hc_\bullet^\k}^0(f,M) &= \Hc_\bullet^\k(I,M) \text{ \hspace{1em}and }& \Hc_\bullet^\k(A,M) &\simeq \colim_n \Fc_{\Hc_\bullet^\k}^n(f,M).
\end{align*}
\begin{lem}\label{lem:filtrationF}
 The quotients $\quot{\Fc_{\Bc_\bullet^\k}^{n+1}(f,M)}{\Fc_{\Bc_\bullet^\k}^n(f,M)}$ and $\quot{\Fc_{\Hc_\bullet^\k}^{n+1}(f,M)}{\Fc_{\Hc_\bullet^\k}^n(f,M)}$ are isomorphic to the complex
 \[
  A^{\otimes n} \otimes B \otimes \Bc_\bullet^\k(I,M)[n+1].
 \]
\end{lem}
\begin{proof}
 We have $\quot{M^{n+1}_p}{M^n_p} \simeq 0$ if $p \leq n$ and $\quot{M^{n+1}_p}{M^n_p} \simeq A^{\otimes n} \otimes B \otimes I^{p-n} \otimes M$ else.
 A rapid computation shows that in the induced differential
 \[
  A^{\otimes n} \otimes B \otimes I^{p+1-n} \otimes M \lra A^{\otimes n} \otimes B \otimes I^{p-n} \otimes M
 \]
 is equal to $\id_{A^{\otimes n} \otimes B} \otimes (-b')$ in the case of $\quot{\Fc_{\Bc_\bullet^\k}^{n+1}(f,M)}{\Fc_{\Bc_\bullet^\k}^n(f,M)}$ and to $\id_{A^{\otimes n} \otimes B} \otimes b'$ in the case of $\quot{\Fc_{\Hc_\bullet^\k}^{n+1}(f,M)}{\Fc_{\Hc_\bullet^\k}^n(f,M)}$.
\end{proof}

\begin{cor}\label{cor:HIHA}
If $M$ is H-unitary as an $I$-bimodule, then
\[
 0 \simeq \Bc_\bullet^\k(I,M) \simeq \Bc_\bullet^\k(A,M) \hspace{2em} \text{and} \hspace{2em} \Hc_\bullet^\k(I,M) \simeq \Hc_\bullet^\k(A,M).
\]
\end{cor}

\begin{cor}\label{cor:BmodHunital}
Let $N$ be a connective left $B$-dg-module.
 If $I$ is H-unital then
 \[
  \Hc_\bullet^\k(A, N \otimes I) \simeq \Hc_\bullet^\k(I, N \otimes I) = \Bc_\bullet^\k(I, N \otimes I) \simeq 0.
 \]
\end{cor}

\begin{proof}
 The first equivalence is an application of the above corollary. The equality follows from the trivial observation that the left action if $I$ on $N$ is trivial. The last equivalence is implied by the fact that $I$ is H-unital, together with \autoref{rmq:MtimesAHunital}.
\end{proof}

\paragraph{A filtration from $\Bc_\bullet^\k(A,B)$ to $\Bc_\bullet^\k(B)$ and from $\Hc_\bullet^\k(A,B)$ to $\Hc_\bullet^\k(B)$}\label{par:quotfiltration}
For $n,p \in \N$, we set
\[
 N^n_p := \begin{cases}
           B^{\otimes p+1} & \text{if } p \leq n \\
           B^{\otimes n+1} \otimes A^{\otimes p-n} & \text{else.}
          \end{cases}
\]
We define $\Qc_{\Bc_\bullet^\k}^n(f)$ (resp. $\Qc_{\Hc_\bullet^\k}^n(f)$) as the quotient of $\Bc_\bullet^\k(A,B)$ (resp. of $\Hc_\bullet^\k(A,B)$) given by
\begin{align*}
  \Qc_{\Bc_\bullet^\k}^n(f) \colon& \hspace{1em} \cdots \lra N^n_p \lra^{-b'} N^n_{p-1} \lra^{-b'} \cdots \lra^{-b'} N^n_1 \lra^{-b'} N^n_0 \lra 0 \\
  \text{\big(resp. }\Qc_{\Hc_\bullet^\k}^n(f) \colon& \hspace{1em} \cdots \lra N^n_p \lra^{b} N^n_{p-1} \lra^{b} \cdots \lra^{b} N^n_1 \lra^{b} N^n_0 \lra 0 \text{ \big). }
\end{align*}

We have $\Qc^0_{\Bc_\bullet^\k}(f) \simeq \Bc_\bullet^\k(A,B)$ and $\Qc_{\Hc_\bullet^\k}^0(f) \simeq \Hc_\bullet^\k(A,B)$, as well as
\[
 \colim_n \Qc_{\Bc_\bullet^\k}^n(f) \simeq \Bc_\bullet^\k(B) \hspace{2em} \text{and}
 \hspace{2em}  \colim_n \Qc_{\Hc_\bullet^\k}^n(f) \simeq \Hc_\bullet^\k(B).
\]

\begin{lem}\label{lem:kernelsfiltration}
 For any $n \in \N$, we have:
 \begin{align*}
  \ker\left(\Qc^n_{\Bc_\bullet^\k}(f) \to \Qc^{n+1}_{\Bc_\bullet^\k}(f)\right) &\simeq \Bc_\bullet^\k(A, B^{\otimes n+1} \otimes I)[n+1]\\
  \ker\left(\Qc^n_{\Hc_\bullet^\k}(f) \to \Qc^{n+1}_{\Hc_\bullet^\k}(f)\right) &\simeq \Hc_\bullet^\k(A, B^{\otimes n+1} \otimes I)[n+1].
 \end{align*}
\end{lem}

\begin{proof}
 We have $K_p^n := \ker(N^n_p \to N^{n+1}_p) \simeq 0$ if $p \leq n$ and $K_p^n \simeq B^{\otimes n+1} \otimes I \otimes A^{\otimes p-n-1}$ else.
 The differential $b'$ induces the differential $K_{p+1}^n \to K_p^n$ given on a homogeneous tensor $b_0\otimes \cdots \otimes b_{n} \otimes i_{n+1} \otimes a_{n+2}\otimes \cdots \otimes a_{p+1}$ by the formula
 \begin{multline*}
 \pm b_0\otimes \cdots \otimes b_{n} \otimes i_{n+1} a_{n+2}\otimes a_{n+3} \otimes \cdots \otimes a_{p+1} \\+
  \sum_{j=n+2}^{p} \pm b_0\otimes \cdots \otimes b_{n} \otimes i_{n+1} \otimes a_{n+2}\otimes \cdots \otimes a_ja_{j+1} \otimes \cdots \otimes a_{p+1}.
 \end{multline*}
The other terms, for $0 \leq j \leq n$, are cancelled because of the vanishing of the composite $I \to A \to B$. The conclusion follows.
\end{proof}

\begin{cor}\label{cor:HAHB}
 If $I$ is H-unital, then
 \[
  \Bc_\bullet^\k(A,B) \simeq \Bc_\bullet^\k(B) \hspace{2em} \text{and} \hspace{2em} \Hc_\bullet^\k(A,B) \simeq \Hc_\bullet^\k(B).
 \]
\end{cor}

\begin{proof}
 It follows from \autoref{lem:kernelsfiltration} and \autoref{cor:BmodHunital} that the filtration from $\Hc_\bullet^\k(A,B)$ to $\Hc_\bullet^\k(B)$ we just introduced is quasi-constant. The case of $\Bc_\bullet^\k$ is similar.
\end{proof}

\paragraph{The proof of \autoref{thm:wodzicki}.}
We start with an extension $I \to A \to B$ of connective dg-algebras. Up to suitable replacements, we can assume that $f \colon A \to B$ is a degree-wise surjective, and that $I$ is its kernel.
Consider the commutative diagram
\[
 \begin{tikzcd}
  \Hc_\bullet^\k(I) \arrow[d, "\alpha"] \arrow[dr] & & \\
  \Hc_\bullet^\k(A,I) \arrow[r] & \Hc_\bullet^\k(A) \ar[r] \ar[dr] & \Hc_\bullet^\k(A,B) \ar[d, "\beta"]\\
  & & \Hc_\bullet^\k(B).
 \end{tikzcd}
\]
The functor $\Hc_\bullet^\k(A,-)$ preserves fiber sequences, and therefore the horizontal sequence is a fiber sequence. From \autoref{cor:HIHA} and \autoref{cor:HAHB}, we deduce that $\alpha$ and $\beta$ are quasi-isomorphisms. In particular the diagonal sequence is a fiber sequence. Similarly, we show that the sequence
\[
 \Bc_\bullet^\k(I) \to \Bc_\bullet^\k(A) \to \Bc_\bullet^\k(B)
\]
is a fiber sequence.
The result then follows from the definitions of $\HH_\bullet^\k$ and $\HC_\bullet^\k$.

\paragraph{Invariance under quasi-isomorphisms}

Let us record for future use that all the constructions of the previous paragraphs are well-behaved with respect to quasi-isomorphisms.

\begin{lem} The following statements hold.
\begin{enumerate}
 \item Let $A \to A'$ be a quasi-isomorphism in $\MCdgAlg_\k^{\mathrm{nu},\leq 0}$ and let $M_1 \to M_2$ be a quasi-isomorphism of connective $A'$-bimodules.
 The induced morphisms
 \begin{align*}
  \Bc_\bullet^\k(A,M_1) &\to \Bc_\bullet^\k(A',M_1) \to \Bc_\bullet^\k(A',M_2) \\
  \text{and~~~} \Hc_\bullet^\k(A,M_1) &\to \Hc_\bullet^\k(A',M_1) \to \Hc_\bullet^\k(A',M_2)
 \end{align*}
 are quasi-isomorphisms.
 In particular, we have quasi-isomorphisms $\Bc_\bullet^\k(A) \to \Bc_\bullet^\k(A,A') \to \Bc_\bullet^\k(A')$ and $\Hc_\bullet^\k(A) \to \Hc_\bullet^\k(A,A') \to \Hc_\bullet^\k(A')$ as well as $\HH_\bullet^\k(A) \to \HH_\bullet^\k(A')$ and $\HC_\bullet^\k(A) \to \HC_\bullet^\k(A')$.
 
 \item Let $f \colon A \to B$ and $g \colon A' \to B'$ be two fibrations in $\MCdgAlg_\k^{\mathrm{nu},\leq 0}$ and let $A \to A'$ and $B \to B'$ be two quasi-isomorphisms commuting with $f$ and $g$. Denote by $I$ and $I'$ the kernels of $f$ and $g$, respectively (so that the induced morphism $I \to I'$ is a quasi-isomorphism too). Then
 \begin{assertions}
  \item The following induced morphisms are quasi-isomorphisms:
  \[
   \Qc_{\Bc_\bullet^\k}^n(f) \to \Qc_{\Bc_\bullet^\k}^n(g), \hspace{3em} \Qc_{\Hc_\bullet^\k}^n(f) \to \Qc_{\Hc_\bullet^\k}^n(g).
  \]
  \item For any morphism of connective $A'$-bimodules $M_1 \to M_2$, the induced morphisms
  \begin{align*}
  \Fc_{\Bc_\bullet^\k}^{n}(f,M_1) &\to \Fc_{\Bc_\bullet^\k}^{n}(g,M_1) \to \Fc_{\Bc_\bullet^\k}^{n}(g,M_2) \\ \text{and}\hspace{1em}
  \Fc_{\Hc_\bullet^\k}^{n}(f,M_1) &\to \Fc_{\Hc_\bullet^\k}^{n}(g,M_1) \to \Fc_{\Hc_\bullet^\k}^{n}(g,M_2)
  \end{align*}
 are quasi-isomorphisms.
 \end{assertions}
\end{enumerate}
\end{lem}

\begin{proof}
 Those are standard arguments: all the complexes at hand are defined as the total space of a simplicial object. In particular, they come with a canonical filtration whose graded parts are of one of the following forms (up to the obvious notational changes)
 \[
  M \otimes_\k A^{\otimes n} \hspace{2em}\text{or}\hspace{2em} M \otimes_\k A^{\otimes n} \otimes_\k I^{p-n} \hspace{2em}\text{or}\hspace{2em} B^{\otimes n+1} \otimes_\k A^{\otimes p-n}
 \]
 Since $\k$ is a field, the induced morphisms between those graded parts are quasi-isomorphisms and the result follows.
\end{proof}

\begin{cor}
 The functors $\Bc_\bullet^\k$, $\Hc_\bullet^\k$, $\Qc_{\Bc_\bullet^\k}^n$, $\Qc_{\Hc_\bullet^\k}^n$, $\Fc_{\Bc_\bullet^\k}^n$, $\Fc_{\Hc_\bullet^\k}^n$, $\HH_\bullet^\k$ and $\HC_\bullet^\k$ descend to $\infty$-functors (that we will denote the same) between the appropriate $\infty$-categories localized along quasi-isomorphisms.
\end{cor}

\subsection{Relative cyclic homology and K-theory}\label{subsec:relK}
In this subsection, we recall the fundamental notions of $\K$-theory and of the equivariant Chern character, at least in the relative setting.

We consider the (connective) $\K$-theory functor as an $\infty$-functor $\dgAlg_\k^{\leq 0} \to \Sp_{\geq 0}$.
Let $(\dgAlg_\k^{\leq 0})^{\Delta^1}_\mathrm{nil}$ denote the $\infty$-category of morphisms $A \to B$ that are surjective with nilpotent kernel at the level of $\homol^0$.

\begin{df}
 The relative $\K$-theory functor is the $\infty$-functor
 \[
  \bK \colon (\dgAlg_\k^{\leq 0})^{\Delta^1}_\mathrm{nil} \to \Sp_{\geq 0}
 \]
 given on a morphism $f \colon A \to B$ by the homotopy fiber $\bK(f) = \hofib(\K(A) \to \K(B))$.
\end{df}
Since the map $\homol_0A \to \homol_0B$ is surjective with nilpotent kernel, the map $\K_1(A) \to \K_1(B)$ is also surjective while the map $\K_0(A) \to \K_0(B)$ is an isomorphism. In particular, the spectrum $\bK(f)$ is connected.

\begin{rmq}\label{rmk:relKnc}
 Using Bass' exact sequence, one can easily show (see for instance \cite[\S 2.8]{beilinson:relK}) that relative connective $\K$-theory and relative non-connective $\K$-theory are equivalent:
 \[
  \bK(f) := \hofib(\K(A) \to \K(B)) \simeq \hofib(\K^\mathrm{nc}(A) \to \K^\mathrm{nc}(B)) \in \Sp.
 \]
\end{rmq}

\begin{df}
 The relative cyclic homology functor is the $\infty$-functor 
 \[
  \bHC_\bullet^\k \colon (\dgAlg_\k^{\leq 0})^{\Delta^1}_\mathrm{nil} \to \dgMod_\k^{\leq 0}
 \]
 given on $f \colon A \to B$ by $\bHC_\bullet^\k(f) := \hofib(\HC_\bullet^\k(A) \to \HC_\bullet^\k(B))$.
\end{df}

\begin{rmq}
 Goodwillie's definition of relative $\K$-theory and cyclic homology in \cite{goodwillie:tangentk} differs from ours by a shift of $1$.
\end{rmq}

The main result of \cite{goodwillie:tangentk} states that the Chern character induces an equivalence $\bch \colon \bK \wedge \Q \to \bHC_\bullet^\Q[1]$.
The construction of the relative Chern character from the absolute one is straigthforward. We will however need a more hands-on construction in \autoref{sec:tracemap}. The two constructions have been proven to coincide by Cortiñas and Weibel in \cite{cortinasweibel:relative} (see \autoref{rmq:relChern} below).

In order to construct our relative Chern character, we will need some explicit model for relative $\K$-theory of connective dg-algebras over $\k$ (containing $\Q$).
For convenience, we will work with the equivalent model of simplicial $\k$-algebras. We denote by $\sAlg_\k$ the category of simplicial (unital) algebras over $\k$, endowed with its standard model structure.

\paragraph{Matrix groups and $\K$-theory}

We start by recalling notions from \cite{waldhausen:ktheory} (see also \cite{goodwillie:tangentk}).

For $A \in \sAlg_\k$, we denote by $\mathrm{M}_n(A)$ the simplicial set obtained by taking $n \times n$-matrices level-wise. Finally, we define the group of invertible matrices as the pullback
\[
 \begin{tikzcd}
  \GL_n(A) \arrow[r] \cart \arrow[d] & \mathrm{M}_n(A) \arrow[d] \\
  \GL_n(\pi_0A) \arrow[r] & \mathrm{M}_n(\pi_0A).
 \end{tikzcd}
\]
We have $\pi_0(\GL_n(A)) \simeq \GL_n(\pi_0A)$ and $\pi_i(\GL_n(A)) \simeq \mathrm{M}_n(\pi_i A)$ for $i \geq 1$. In particular, this construction preserves homotopy equivalences. The simplicial set $\GL_n(A)$ is a group-like simplicial monoid and we denote by $\BGL_n(A)$ its classifying space.
Finally, we denote by $\BGL_\infty(A)$ the colimit $\colim_n \BGL_n(A)$.

Applying Quillen's plus construction to $\BGL_\infty(A)$ yields a model for $\K$-theory, so that there is an equivalence
\[
 \K_0 \times \BGL_\infty(-)^+ \simeq \infloop \K.
\]
Based on this equivalence, we will build in the next paragraph a model for relative $\K$-theory using relative Volodin spaces.

\paragraph{Relative Volodin construction}
The Volodin model for the $\K$-theory of rings first appeared in \cite{volodin:ktheory} in the absolute case. The relative version seem to originate from an unpublished work of Ogle and Weibel. It can also be found in \cite{loday:hc}. We will need a version of that construction for simplicial algebras over $\k$.

We fix a fibration $f \colon A \to B$ in $\sAlg_\k$ such that the induced morphism $\pi_0(A) \to \pi_0(B)$ is surjective. In particular, the morphism $f$ is level-wise surjective. We denote by $I$ its kernel and we assume that $\pi_0I$ is nilpotent in $\pi_0A$.
\begin{df}
 Let $n \geq 1$ and let $\sigma$ be a partial order on the set $\{1, \dots, n\}$. We denote by $\mathrm{T}^\sigma_n(A,I)$ the sub simplicial set of $\mathrm{M}_n(A)$ given in dimension $p$ by the subset of $\mathrm{M}_n(A_p)$ consisting of matrices of the form $1+(a_{ij})$ with $a_{ij} \in I_p$ if $i$ is not lower than $j$ for the order $\sigma$.
\end{df}
 
 As a simplicial set, $\mathrm{T}^\sigma_n(A,I)$ is isomorphic to a simplicial set of the form $A^\alpha \times I^\beta$ with $\alpha$ and $\beta$ are integers depending on $\sigma$, such that $\alpha + \beta = n^2$. In particular, this construction is homotopy invariant.
Moreover, the map $\mathrm{T}^\sigma_n(A,I) \to \mathrm{M}_n(A)$ factors through $\GL_n(A)$. Actually, $\mathrm{T}^\sigma_n(A,I)$ is a simplicial subgroup of the simplicial monoid $\GL_n(A)$.

\begin{df}
 We define the relative Volodin space $X(A,I)$ as the union 
 \[
  X(A,I) := \bigcup_{n, \sigma} \B\mathrm{T}^\sigma_n(A,I) \subset \BGL_\infty(A).
 \]
\end{df}
\noindent The group $\pi_1(X(A,I))$ contains as a maximal perfect subgroup the group $E(\pi_0(A))$ and we apply the plus construction to this pair.

\begin{prop}
The inclusion $X(A,I) \to \BGL_\infty(A)$ induces a fiber sequence
\[
 X(A,I)^+ \to \BGL_\infty(A)^+ \to \BGL_\infty(B)^+.
\]
In particular, we have a (functorial) equivalence
\[
 X(A,I)^+ \simeq \infloop \bK(f).
\]
\end{prop}

\begin{proof}
This is a classical argument, that can be found in \cite[\S11.3]{loday:hc} for rings. We simply extend it to simplicial algebras.
 We start by drawing the following commutative diagram
\[
 \begin{tikzpicture}[baseline=0em,xscale=0.25em,yscale=-.16em]
  \node (00) at (0,0) {$X(A,0)$};
  \node (10) at (1,0) {$X(A,I)$};
  \node (20) at (2,0) {$\Omega^{\infty}\bK(f)$};
  \node (01) at (0,1) {$X(A,0)$};
  \node (11) at (1,1) {$\BGL_\infty(A)$};
  \node (21) at (2,1) {$\BGL_\infty(A)^+$};
  \node (02) at (0,2) {$0$};
  \node (12) at (1,2) {$\BGL_\infty(B)^+$};
  \node (22) at (2,2) {$\BGL_\infty(B)^+$};
  
  \node (A) at (0, -0.4) {(a)};
  \node (B) at (1, -0.4) {(b)};
  \node (C) at (2, -0.4) {(c)};
  \node (1) at (-0.6, 0) {(1)};
  \node (2) at (-0.6, 1) {(2)};
  \node (3) at (-0.6, 2) {(3)};
  
  \path[commutative diagrams/.cd, every arrow, every label]
(00) edge (10) edge (01)
(10) edge (20) edge (11)
(20) edge (21)
(01) edge (11) edge (02)
(11) edge (21) edge (12)
(02) edge (12)
(12) edge (22)
(21) edge (22);
\end{tikzpicture}
\]

The space $X(A,0)$ is acyclic (see \cite{suslin:eqK} for the discrete case, the simplicial case being deduced by colimits). It follows (by the properties of the plus construction), that the row (2) is a fibration sequence. Obviously, so are the row (3) and the columns (a) and (c).
It now suffices to show that column (b) is a fibration sequence.
Consider the commutative diagram
 \[
  \begin{tikzcd}
   X(A,I) \ar[r] \ar[d] \ar[dr, phantom, "\displaystyle (\tau)"] & \BGL_\infty(A) \ar[d, "p"] \ar[r] & \BGL_\infty(B)^+ \ar[d] \\
   X(B,0) \ar[r] & \BGL_\infty(B) \ar[r] & \BGL_\infty(B)^+.
  \end{tikzcd}
 \]
 The square $(\tau)$ is homotopy Cartesian. Indeed, the morphism $p$ is induced by a point-wise surjective fibration of group-like simplicial monoids (recall that we assumed $\pi_0(I)$ to be nilpotent). It is therefore a fibration. We can now see that the diagram is cartesian on the nose by looking at its simplices.
 The bottom row is a fibration sequence. It follows that the top row (which coincides with column (b)) is also a fibration sequence.
 
 As a consequence, row (1) induces a fibration sequence in homology. Since $X(A,0)$ is acyclic, the homologies of $X(A,I)$ and of $\Omega^{\infty}\bK(f)$ are isomorphic. It follows that $X(A,I)^+$ and $\infloop \bK(f)$ are homotopy equivalent.
\end{proof}

\paragraph{Malcev's theory}
In order to construct our relative Chern character $\bch \colon \bK \wedge \Q \to \bHC_\bullet^\Q[1]$, it is now enough to relate the homology of the relative Volodin spaces $X(A,I)$ with cyclic homology.
This step uses Malcev's theory, that relates homology of nilpotent uniquely divisible groups (such as $\mathrm{T}^\sigma_n(A,I)$ for $A$ discrete) with the homology of an associated nilpotent Lie algebra.
The original reference is \cite{malcev:lie}.

For simplicity, we will only work with Lie algebras of matrices. We assume for now that $A$ is a discrete unital $\Q$-algebra. Let $\n \subset \gl_n(A)$ be a nilpotent sub-Lie algebra (for some $n$). Denote by $N$ the subgroup $N := \exp(\n) \subset \GL_n(A)$.
We have the following proposition (for a proof, 
we refer to \cite[Theorem 5.11]{suslinwodzicki}).
\begin{prop}[Malcev, Suslin-Wodzicki]\label{prop:malcev}
There is a quasi-isomorphism
\[
 \Q[\B N] := \mathrm C_\bullet(\B N, \Q) \to \CE_\bullet^\Q(\n)
\]
functorial in $\n$, where $\B N$ denotes the classifying space of $N$.
Moreover, this quasi-isomorphism is compatible with the standard filtrations on those complexes.
\end{prop}
The construction of this quasi-isomorphism is based on two statements. First, the completion of the algebras $\mathcal U(\n)$ and $\Q[N]$ along their augmentation ideals are isomorphic and, second, the standard resolutions of $\Q$ as a trivial module on those algebras are compatible.
It follows that this quasi-isomorphism is actually compatible with the standard filtrations on both sides.

Fix $I \subset A$ a nilpotent ideal. For $n \in \N$ and $\sigma$ a partial order on $\{1, \dots, n\}$, we denote by $\mathrm{t}^\sigma_n(A,I) \subset \gl_n(A)$ the (nilpotent) Lie algebra of matrices $(a_{ij})$ such that $a_{ij} \in I$ if $i$ is not smaller that $j$ for the partial order $\sigma$.
We have then by definition $\mathrm{T}^\sigma_n(A,I) = \exp(\mathrm{t}^\sigma_n(A,I))$ and therefore a quasi-isomorphism 
 $\tau \colon \Q[\B \mathrm{T}^\sigma_n(A,I)] \to \CE_\bullet^\Q(\mathrm{t}^\sigma_n(A,I))$ functorial in $A$ and $I$.
 
Return now to the general case of a morphism of simplicial rings $f \colon A \to B$ such that $\pi_0 A \to \pi_0 B$ is surjective with nilpotent kernel. First, we replace $A$ so that $A \to F$ is a filtration (i.e. is levelwise surjective) and $\ker(A_0 \to B_0)$ is nilpotent. Using the monoidal Dold--Kan correspondence (see \cite{schwedeshipley:monoidalDK}), we can further assume $A_n$ (resp. $B_n$) to be a nilpotent extension of $A_0$ (resp. $B_0$). All in all, we have replaced $f$ with an equivalent morphism, which is levelwise surjective and such that the kernel $I$ is (levelwise) nilpotent.
Both functors $\Q[\B \mathrm{T}^\sigma_n(-,-)]$ and $\CE_\bullet^\Q(\mathrm{t}^\sigma_n(-,-))$ preserve geometric realizations and we therefore get a (functorial) quasi-isomorphism
\[
 \tau \colon \Q[\B \mathrm{T}^\sigma_n(A,I)] \to \CE_\bullet^\Q(\mathrm{t}^\sigma_n(A,I))
\]
by applying $\tau$ level-wise.

\paragraph{The relative Chern character}\label{par:chern}
The usual construction (used among others by Goodwillie) is based on the absolute Chern character $\mathrm{ch} \colon \K \to \HC^\Q_\bullet(-)[1]$. It is a (functorial) morphism of spectra $\bK(f) \to \bHC_\bullet^\Q(f)[1]$. We will only need the induced relative (and $\Q$-linear) version:

\begin{df}
 We will denote by $\bch$ the ($\Q$-linear) relative Chern character
 \[
  \bch \colon \bK(f) \wedge \Q \to \bHC_\bullet^\Q(f)[1].
 \]
\end{df}

\begin{thm}[Goodwillie \cite{goodwillie:tangentk}]\label{thm:goodwillie}
 The morphism $\bch$ is a quasi-isomorphism.
\end{thm}

In order to compare a tangent map to the generalized trace in \autoref{sec:tracemap}, we will need another description of the relative Chern character.
Using the natural inclusion $\mathrm{t}^\sigma_n(A,I) \to \gl_\infty(A)$ and the generalized trace map, we get a morphism
\[
 \CE_\bullet^\Q(\mathrm{t}^\sigma_n(A,I)) \to \CE_\bullet^\Q(\gl_\infty(A)) \to \HC_\bullet^\Q(A)[1]
\]
whose image lies in the subcomplex $\bHC_\bullet^\Q(f)[1]$. We find
\[
 \Q[\B \mathrm{T}^\sigma_n(A,I)] \to^\sim \CE_\bullet^\Q(\mathrm{t}^\sigma_n(A,I)) \to^\Tr \bHC_\bullet^\Q(f)[1].
\]
Taking the colimit on $n$ and $\sigma$, we find a version of the relative Chern character
\[
 \bch^\Q \colon \Q[\infloop \bK(f)] \simeq \Q[X(A,I)] \to \bHC_\bullet^\Q(f)[1].
\]
This version a priori does not descend to a morphism of spectra. However, Cortiñas and Weibel proved that $\bch$ induces a morphism $\Q[\infloop \bK(f)] \to \bHC_\bullet^\Q(f)[1]$ that is functorially homotopic to our $\bch^\Q$ (see \cite{cortinasweibel:relative}).
\begin{rmq}\label{rmq:relChern}
Cortiñas and Weibel only consider the case of discrete algebras. If $f \colon A \to B$ is a surjective morphism of simplicial algebras, we can reduce to the discrete case. Indeed, writing $f$ as a geometric realization of discrete $f_n \colon A_n \to B_n$ and setting $I_n = \ker(f_n)$, we find $\Q[X(A,I)] \simeq \colim_{[n]} \Q[X(A_n,I_n)]$ as well as $\bHC_\bullet^\Q(f) \simeq \colim_{[n]} \bHC_\bullet^\Q(f_n)$.
 The relative Chern characters evaluated on $I \subset A \to B$ are then (both) determined by their value on discrete algebras, and therefore coincide using Cortiñas and Weibel's result.
\end{rmq}

\paragraph{The case of H-unital algebras}\label{par:goodwillie_for_H_unital}
In the proof of our main result, we will use a Goodwillie theorem for H-unital algebras. To extend \autoref{thm:goodwillie} to this context, we will need the following result of Suslin and Wodzicki \cite{suslinwodzicki} (see \cite{tamme:excisionrevisited} for its generalization to simplicial $\Q$-algebras)
\begin{thm}[Suslin-Wodzicki, Tamme]
 If $I \to A \to B$ is an extension of (possibly non-unital) simplicial $\Q$-algebras and $I$ is H-unital, then the induced sequence
 \[
  \K(I)\wedge \Q \to \K(A)\wedge \Q \to \K(B)\wedge \Q
 \]
 is a fiber sequence of connective spectra. In particular, if either $A \to B$ admits a section or $I$ is nilpotent, it is also a cofiber sequence.
\end{thm}
We fix a nilpotent extension $f \colon A \twoheadrightarrow B$ of H-unital simplicial $\k$-algebras. We get a commutative diagram whose rows and columns are fiber and cofiber sequences
\[
\begin{tikzcd}
  \bK(f)\wedge \Q \ar{r}{\simeq} \ar{d} & \bK(\k \ltimes f)\wedge \Q \ar{r} \ar{d} & 0 \ar{d} \\
  \K(A)\wedge \Q \ar{r} \ar{d} & \K(\k \ltimes A)\wedge \Q \ar{r} \ar{d} & \K(\k)\wedge \Q \ar{d}[sloped, above]{\simeq} \\
  \K(B)\wedge \Q \ar{r} & \K(\k \ltimes B)\wedge \Q \ar{r} & \K(\k)\wedge \Q, 
\end{tikzcd}
\]
where the functor $\k \ltimes -$ formally adds a ($\k$-linear) unit\footnote{So $\k \ltimes -$ is the left adjoint to the forgetful functor $\sAlg_\k \to \sAlg_\k^\mathrm{nu}$}. We get a natural equivalence 
\[
 \bK(f)\wedge \Q \simeq \bK(\k \ltimes f)\wedge \Q.
\]
Similarly, using \autoref{thm:wodzicki}, we have $\bHC^\Q_\bullet(f)[1] \simeq \bHC^\Q_\bullet(\k \ltimes f)[1]$. \autoref{thm:goodwillie} thus leads to the following

\begin{cor}\label{cor:goodwillieHunital}
 There is a functorial equivalence
\[
 \bch \colon \bK(f)\wedge \Q \to^\sim \bHC^\Q_\bullet(f)[1],
\]
in the more general case of $f \colon A \to B$ a nilpotent extension of H-unital simplicial $\Q$-algebras (that coincides with the relative Chern character when both $A$ and $B$ are unital).
\end{cor}

\section{Formal deformation problems}

Infinitesimal deformations of algebraic objects can be encoded by a tangential structure on the moduli space classifying those objects. In \cite{pridham:deformation} and \cite{lurie:dagx}, Pridham and Lurie established an equivalence between so-called formal moduli problems and differential graded Lie algebras.
This section starts by recalling Pridham and Lurie's works. We then establish some basic facts about abelian or linear formal moduli problems.

\subsection{Formal moduli problems and dg-Lie algebras}

\begin{df}\label{df:schlessinger}
 Let $\dgArt_\k \subset \quot{\cdga_\k}{\k}$ denote the full subcategory of augmented connective $\k$-cdga's spanned by Artinian\footnote{Recall that $A$ is called Artinian if its cohomology is finite dimensional over $\k$, and if $\homol^0(A)$ is a local ring.} ones.
 Let $\Cc$ be an $\infty$-category with finite limits.
 
 We say that a functor $F \colon \dgArt_\k \to \Cc$ satisfies the \ref{cond:schlessinger} if
 \begin{enumerate}[label=(S),ref=Schlessinger condition {\rm (S)}]
  \item \label{cond:schlessinger} $\left\{\parbox{0.9\textwidth}{%
 \begin{enumerate}[label=(S\arabic*),ref=Schlessinger condition {\rm (S\arabic*)}]
  \item\label{cond:S1} The object $F(\k)$ is final in $\Cc$.
  \item\label{cond:S2} \parbox{0.8\textwidth}{For any map $A \to B \in \dgArt_\k$ that is surjective on $\homol^0$, the induced morphism $F(\k \times_B A) \to F(\k) \times_{F(B)} F(A)$ is an equivalence.}
 \end{enumerate}}\right.$
 \end{enumerate}
\end{df}
 
 \begin{df}
 A pre-FMP (pre-formal moduli problem) is a functor $\dgArt_\k \to \sSets$.
 We denote by $\PFMP_\k$ their category.
 A FMP (formal moduli problem) is a pre-FMP $F$ satisfying the \ref{cond:schlessinger}. We denote by $\FMP_\k$ the category of formal moduli problems, and by $\incl \colon \FMP_\k \to \PFMP_\k$ the inclusion functor.
\end{df}

\begin{ex}\label{ex:artinstacksareFMP}
 Let $X$ be an Artin (and possibly derived) stack and $x \in X$ be a $\k$-point. The functor $B \mapsto \{ x\} \times_{X(\k)} X(B)$ defined on Artinian dg-algebras satisfies the Schlessinger condition (see \cite{toen:hagii} for instance). It thus defines a formal moduli problem.
\end{ex}

The category $\PFMP_\k$ is presentable, and the full subcategory $\FMP_\k$ is strongly reflexive. In particular, the inclusion $\incl  \colon \FMP_\k \to \PFMP_\k$ admits a left adjoint.
\begin{df}
 We denote by $\form \colon \PFMP_\k \to \FMP_\k$ the left adjoint to the inclusion $\incl$. We call it the formalization functor.
\end{df}

\begin{df}
 A shifted dg-Lie algebra over $\k$ is a complex $V$ together with a dg-Lie algebra structure on $V[-1]$. We denote by $\dgLie_\k^{\Omega}$ the $\infty$-category of shifted dg-Lie algebras.
\end{df}

\begin{rmq}
 The notation $\Omega$ is here to remind us the shift in the Lie structure. Note that shifting a complex $V$ by $-1$ amounts to computing its (pointed) loop space $\Omega V \simeq V[-1]$.
\end{rmq}

\begin{thm}[Pridham, Lurie]\label{thm:pridhamlurie}
The functor $\T \colon \FMP_\k \to \dgMod_\k$ (computing the tangent complex) factors through the forgetful functor $\dgLie_\k^\Omega \to \dgMod_\k$. In other words, the tangent complex $\T F$ of a formal moduli problem admits a natural shifted Lie structure. Moreover, the functor $F \mapsto \T F$ induces an equivalence
 \[
  \FMP_\k \simeq \dgLie_\k^{\Omega}.
 \]
\end{thm}

\begin{df}
 We denote by $\ell$ the composite functor $\ell := \T \circ \form \colon \PFMP_\k \to \dgLie_\k^\Omega$, and by $\e \colon \dgLie_\k^\Omega \to^\sim \FMP_\k$ the inverse functor $\T^{-1}$.
\end{df}

\begin{ex}\label{ex:tgtcoincides}
 Let $X$ be a smooth scheme and $x \in X$ a $\k$-point. \autoref{ex:artinstacksareFMP} yields an associated formal moduli problem whose tangent complex is $\mathrm T_{X,x}$, the tangent space of $X$ at $x$. The shifted Lie structure is the trivial.
 
 More interestingly, if $G$ is a smooth group scheme, then its classifying stack $\B G$ is an Artin stack. The tangent complex to the associated formal moduli problem is $\g[1] = \mathrm T_{G, 1}[1]$, the shift of the Lie algebra of $G$. The shifted Lie algebra structure on $\g[1]$ is the usual Lie algebra structure on $\g$. This will be discussed in more detail in \autoref{sec:tracemap} below.
\end{ex}

Let us recall briefly how the equivalence in \autoref{thm:pridhamlurie} is constructed.
Consider the Chevalley-Eilenberg cohomological functor
\[
 \CE^\bullet_\k \colon \dgLie\op_\k \to \quot{\cdgaunbounded_\k}{\k}.
\]
It admits a left adjoint, denoted by $\Dd_\k$, so that for any $L$ and $B$, we have
\[
 \Map_{\dgLie_\k}(L, \Dd_\k(B)) \simeq \Map_{\cdgaunbounded_\k^\mathrm{aug}}(B, \CE^\bullet_\k(L)).
\]
The equivalence $\e \colon \dgLie_\k^\Omega \to^\sim \FMP_\k$ is then given on $V \in \dgLie_\k^\Omega$ by the formula
\[
 \e(V)(B) := \Map_{\dgLie_\k}(\Dd_\k(B), V[-1]).
\]
Proving this construction indeed defines an equivalence is based on the following key lemma
\begin{lem}[see{ \cite[2.3.5]{lurie:dagx}}]\label{lem:BsimCDB}
 For $B \in \dgArt_\k$, the adjunction morphism 
 \[
	B \to \CE^\bullet_\k(\Dd_\k(B))
 \]
  is an equivalence.
\end{lem}

In what follows, we will need slightly more general versions of formal moduli problems, namely formal moduli problems with values in an $\infty$-category such as complexes of vectors spaces or connective spectra.
\begin{df}\label{df:FMPC}
Let $\Cc$ be an $\infty$-category with all finite limits.
A $\Cc$-valued pre-FMP is a functor $\dgArt_\k \to \Cc$. A $\Cc$-valued formal moduli problem  (or FMP) is a $\Cc$-valued pre-FMP satisfying the \ref{cond:schlessinger}.

We denote by $\PFMP^\Cc_\k$ the category of $\Cc$-valued pre-FMPs, by $\FMP^\Cc_\k$ the category of $\Cc$-valued FMPs. We also denote by $\iC \colon \FMP^\Cc_\k \to \PFMP^\Cc_\k$ the inclusion functor.
\end{df}

\begin{lem}\label{lem:iCleftadjoint}
If $\Cc$ is a presentable $\infty$-category, then $\PFMP_\k^\Cc$ and $\FMP_\k^\Cc$ are presentable categories, and $\iC$ admits a left adjoint.
\end{lem}

\begin{df}
 In the above situation, we will denote by $\LC$ the left adjoint to $\iC$.
\end{df}

\begin{proof}[of \autoref{lem:iCleftadjoint}]
 The presentability of $\PFMP_\k^\Cc$ is \cite[5.5.3.6]{lurie:htt}.
 Remains to prove that $\FMP_\k^\Cc \subset \PFMP_\k^\Cc$ is a strongly reflexive category (see \cite[p.482]{lurie:htt}).
 Fix a cartesian square $(\sigma)$ in $\dgArt_\k$
\[
\begin{tikzcd}
  C \ar[r] \ar[d] & A \ar[d, "f"] \\ \k \ar[r] & B \ar[ul, phantom, "\displaystyle (\sigma)"]
\end{tikzcd}
\]
where $f$ is surjective on $\homol^0$. We denote by $\Dd_{(\sigma)} \subset \PFMP_\k^\Cc$ the full subcategory spanned by functors $F$ mapping $(\sigma)$ to a pullback square.
Since $\FMP_\k^\Cc = \bigcap_{(\sigma)} \Dd_{(\sigma)}$ and because of \cite[5.5.4.18]{lurie:htt}, it suffices to prove that $D_{(\sigma)}$ is strongly reflexive in $\PFMP_\k^\Cc$.
 
Restriction along $(\sigma)$ defines a functor $\sigma^* \colon \PFMP_\k^{\Cc} \to \Fct(K,\Cc)$ where $K = \Delta^1 \times \Delta^1$ is the square.
 The functor $\sigma^*$ admits a left adjoint $\sigma_!$ (namely the left Kan extension functor, see \cite[4.3.3.7]{lurie:htt}). A functor $F$ belongs to $D_{(\sigma)}$ if and only if $\sigma^*(F)$ is a pullback square. By \cite[5.5.4.19]{lurie:htt}, the full subcategory $\Fct(K,\Cc)$ spanned by pullback squares is strongly reflexive. It thus follows from \cite[5.5.4.17]{lurie:htt} that $D_{(\sigma)}$ is strongly reflexive in $\PFMP^\Cc_\k$. 
\end{proof}

We conclude this section by recording a functoriality statement, whose proof is straightforward and left to the reader.

\begin{lem}\label{lem:functoriality}
Let $f \colon \Cc \rightleftarrows \Dd \noloc g$ be an adjunction between presentable $\infty$-categories.
\begin{enumerate}
\item Composing with $g$ induces a functor $g_* \colon \PFMP_\k^\Dd \to \PFMP_\k^\Cc$ that maps formal moduli problems to formal moduli problems.
\item The induced functor $g_* \colon \FMP_\k^\Dd \to \FMP_\k^\Cc$ admits a left adjoint $f_!$ given by the composition $f_! = \form^\Dd \circ f_* \circ \iC$, where $f_* \colon \PFMP_\k^\Cc \to \PFMP_\k^\Dd$ is the functor given by composing with $f$.
\item We have $f_! \circ \LC \simeq \form^\Dd \circ f_*$.
\end{enumerate}
\end{lem}

\subsection{Abelian moduli problems}\label{subsec:abelianfmp}

The moduli problem of concern in this article is constructed from the (connective) $\K$-theory functor. In particular, its values are endowed with an abelian group structure, or equivalently are connective spectra. We shall therefore establish a couple of basic properties of those abelian formal moduli problems.

\paragraph{Abelian formal moduli problems}
Denote by $\FMP_\k^\Ab$ (resp. $\PFMP_\k^\Ab$) the category of abelian group objects (i.e. group-like $\E_\infty$-monoids) in (pre-)formal moduli problems.
We call their objects abelian (pre-)formal moduli problems.

Forgetting the abelian group structure and the free abelian group functor form an adjunction
\[
\FAb \colon \FMP_\k \rightleftarrows \FMP_\k^\Ab \noloc \GAb.
\]
We also have a similar adjunction
\[
 \infsusp \colon \PFMP_\k \rightleftarrows \PFMP_\k^\Ab \noloc \infloop.
\]
Note the categories $\PFMP_\k^\Ab$ and $\FMP^\Ab_\k$ identify with $\PFMP_\k^\Cc$ (resp. $\FMP_\k^\Cc$) for $\Cc = \Sp_{\geq 0}$ the category of connective spectra (see \cite[Rmk. 5.2.6.26]{lurie:halg}). In particular, the above adjunction is given by applying point-wise the functors $\infsusp \colon \sSets \rightleftarrows \Sp_{\geq 0} \noloc \infloop$. Finally, we denote by $\iAb \colon \FMP_\k^\Ab \to \PFMP_\k^\Ab$ the inclusion functor, and by $\LAb$ its left adjoint.

\begin{rmq}
We have a Beck-Chevalley transformation $\form \circ \operatorname{\infloop} \to \GAb \circ \LAb$. It is in general not an equivalence. In particular, an abelian pre-FMPs, such as the $\K$-theory functor, will have two different associated formal moduli problems, and two different tangent Lie algebras (one of which will automatically be abelian, see below). 
\end{rmq}

\paragraph{Abelian dg-Lie algebras}

It follows from \autoref{thm:pridhamlurie} that the category $\FMP^\Ab_\k$ is equivalent to the category $\dgLie_\k^{\Omega,\Ab}$ of abelian group objects in $\dgLie^\Omega_\k$. We shall call objects of $\dgLie_\k^{\Omega,\Ab}$ \emph{abelian dg-Lie algebras}.

Although the following statement is well know to the community, we could not locate a proof in the literature. We therefore provide one.
\begin{prop}\label{prop:abelianlie}
The forgetful functor 
\[
\dgLie_\k^{\Omega,\Ab} \to \dgMod_\k
\]
is an equivalence. In particular $\FMP_\k^\Ab \simeq \dgLie_\k^{\Omega,\Ab} \simeq \dgMod_\k$.
\end{prop}

\begin{df}
 We denote by $\TAb \colon \FMP_\k^\Ab \to^\sim \dgMod_\k$ the equivalence of the above proposition. We denote by $\eAb \colon \dgMod_\k \to^\sim \FMP_\k^\Ab$ its inverse.
\end{df}

\begin{proof}[of the proposition]
For any pointed category, we denote by $\Omega$ its pointed loop space endofunctor.
Denote by $\FMP_\k^{\geq n}$ the category of FMP in $(n-1)$-connective spaces. In particular, we have $\FMP_\k^{\geq 0} = \FMP_\k$.
The inclusion of $(n-1)$-connective spaces into all spaces induces a fully faithful functor $\FMP_\k^{\geq n} \to \PFMP_\k$. We denote by $u_n$ the composite
\[
 u_n \colon \FMP_\k^{\geq n} \lra \PFMP_\k \lra^\form \FMP_\k.
\]
\begin{lem}
 The functor $u_n$ is an equivalence.
\end{lem}
\begin{proof}
For any $\infty$-category $\Cc$ with finite products, denote by $\operatorname{Mon}_{\E_n}^\mathrm{gp}(\Cc)$ the $\infty$-category of group-like $\E_n$-monoids in $\Cc$ (see \cite[Def. 5.2.6.6]{lurie:halg}). By \cite[Thm. 5.2.6.10]{lurie:halg}, the $n$-th loop space defines a (point-wise) equivalence $\Omega^n \colon \FMP_\k^{\geq n} \to \operatorname{Mon}_{\E_n}^\mathrm{gp}(\FMP_\k)$.

It follows from \cite[Prop. 2.15]{blanckatzarkovpandit} that taking the $n$-th loop space also defines an equivalence $\FMP_\k \to \operatorname{Mon}_{\E_n}^\mathrm{gp}(\FMP_\k)$. The inverse first applies point-wise the $n$-fold delooping $\B^n$, and then applies $\form$ to the obtained functor.
The composition
\[
 \FMP_\k^{\geq n} \lra^{\Omega^n} \operatorname{Mon}_{\E_n}^\mathrm{gp}(\FMP_\k) \lra^{\form \circ \B^n} \FMP_\k
\]
is then homotopic to $u_n$, and is an equivalence.
\end{proof}

We continue our proof of \autoref{prop:abelianlie}. The forgetful functor $f \colon \dgLie_\k^{\Omega} \to \dgMod_\k$ commutes with limits (and thus with $\Omega$). We get a commutative diagram, where the leftmost column is obtained by taking the limit of the rows
\[
\begin{tikzcd}[row sep=small, column sep=small]
 \Cc_1 := \lim \FMP_\k^{\geq n} \phantom{:= \Cc_1} \ar[phantom, yshift=0.1em]{d}[sloped]{\simeq}[xshift=1em]{\scriptstyle u_\infty} &
 \hspace{2em} \cdots \ar{r}{\Omega} &
 \FMP_\k^{\geq 2} \ar{r}{\Omega} \ar[phantom, yshift=0.1em]{d}[sloped]{\simeq}[xshift=1em]{\scriptstyle u_2} & \FMP_\k^{\geq 1} \ar{r}{\Omega} \ar[phantom, yshift=0.1em]{d}[sloped]{\simeq}[xshift=1em]{\scriptstyle u_1} & \FMP_\k^{\geq 0} \ar[phantom, yshift=0.1em]{d}[sloped]{\simeq}[xshift=1em]{\scriptstyle u_0}
  \\
  \Cc_2 := \lim \FMP_\k \phantom{:= \Cc_2} \ar[phantom]{d}[sloped]{\simeq}[xshift=1em]{\scriptstyle \T_\infty} & \hspace{2em} \cdots \ar{r}{\Omega} & \FMP_\k \ar{r}{\Omega} \ar[phantom]{d}[sloped]{\simeq}[xshift=1em]{\scriptstyle \T} & \FMP_\k \ar{r}{\Omega} \ar[phantom]{d}[sloped]{\simeq}[xshift=1em]{\scriptstyle \T} & \FMP_\k \ar[phantom]{d}[sloped]{\simeq}[xshift=1em]{\scriptstyle \T}
  \\
  \Cc_3 := \lim \dgLie^\Omega_\k \phantom{:= \Cc_3} \ar{d}{f_\infty} & \hspace{2em} \cdots \ar{r}{\Omega} & \dgLie_\k^\Omega \ar{r}{\Omega} \ar{d}{f} & \dgLie_\k^\Omega \ar{r}{\Omega} \ar{d}{f} & \dgLie_\k^\Omega \ar{d}{f} 
  \\
  \Cc_4 := \lim \dgMod_\k \phantom{:= \Cc_4} & \hspace{2em} \cdots \ar{r}{\Omega}[swap]{\sim} & \dgMod_\k \ar{r}{\Omega}[swap]{\sim} & \dgMod_\k \ar{r}{\Omega}[swap]{\sim} & \dgMod_\k.
\end{tikzcd}
\]
Since the category $\dgMod_\k$ is stable, the projection on the rightmost component is an equivalence $\Cc_4 \simeq \dgMod_\k$.
The category $\Cc_1$ identifies with the category of formal moduli problems in the limit category $\lim (\cdots \to \sSets_{\geq 1} \to \sSets_{\geq 0}) \simeq \Sp_{\geq 0}$. By \cite[Rmk. 5.2.6.26]{lurie:halg}, $\Cc_1$ is equivalent to $\FMP_\k^\Ab$, and therefore $\Cc_3 \simeq \dgLie_\k^{\Omega,\Ab}$. Moreover, the equivalence $\T_\infty$ is homotopic to the equivalence $\FMP_\k^\Ab \simeq \dgLie_\k^{\Omega,\Ab}$ induced directly from $\T \colon \FMP_\k \simeq \dgLie_\k^\Omega$ by taking abelian group objects on both sides.
The projections on the rightmost component $\Cc_1 \to \FMP_\k^{\geq 0} = \FMP_\k$ and $\Cc_3 \to \dgLie_\k^\Omega$ identify with the functors forgetting the abelian group structure, while the functor $f_\infty \colon \dgLie_\k^{\Omega,\Ab} \simeq \Cc_3 \to \Cc_4 \simeq \dgMod_\k$ is the forgetful functor.

Remains to prove that $f_\infty$ is an equivalence. Since $f$ is conservative, so is $f_\infty$. The functor $f_\infty$ is the limit of functors with left adjoints, and therefore it admits a left adjoint $g_\infty$. Denote by $g$ the left adjoint of $f$. Up to a shift, the functor $g$ identifies with the free Lie algebra functor:
\[
 g(V) \simeq \mathrm{FreeLie}(V[-1]) [1].
\]
For a fixed $V \in \dgMod_\k \simeq \Cc_4$, the adjunction unit $V \to f_\infty \circ g_\infty(V)$ identifies with the canonical map
\[
\phi_V \colon  V \to \colim_n \Omega^{n}(g(V[n])).
\]
The full subcategory of $\dgMod_\k$ spanned by $V$'s such that $\phi_V$ is an equivalence is stable under filtered colimits. We may thus assume $V$ to be perfect, concentrated in (cohomological) degrees lower than some integer $m$.
Fix $i \in \Z$. For $n > m + i$, the cohomology group $\homol^i(\Omega^{n}(g(V[n])))$ is independent of $n$ and isomorphic to $\homol^i(V)$. In particular, the map $\phi_V$ is a quasi-isomorphism.
\end{proof}

\begin{df}
 We denote by $\CEO \colon \dgLie_\k^\Omega \to \dgMod_\k$ the functor mapping a shifted dg-Lie algebra $L$ to the Chevalley-Eilenberg complex of its shift $\CE_\bullet(L[-1])$.
 
 We denote by $\bCEO \colon \dgLie_\k^\Omega \to \dgMod_\k$ the functor mapping a shifted dg-Lie algebra $L$ to the reduced Chevalley-Eilenberg complex of its shift $\bCE_\bullet(L[-1])$.
 \end{df}
\begin{df}
 We denote by $\theta \colon \dgMod_\k \simeq \dgLie_\k^{\Omega,\mathrm{Ab}} \to \dgLie_\k^{\Omega}$ the functor forgetting the abelian group structure.
\end{df}

\begin{lem}
 The functor $\theta \colon \dgMod_\k \to \dgLie_\k^\Omega$ identifies with the functor mapping a complex to itself with the trivial bracket.
 As a consequence, $\bCEO$ is left adjoint to $\theta$.
\end{lem}
\begin{proof}
Denote by $h(V)$ the (shifted) dg-Lie algebra with trivial bracket built on $V \in \dgMod_\k$. 
 Since $g$ is given as a free (shifted) dg-Lie algebra, there is a canonical morphism $g(V) \to h(V)$ given by collapsing the free brackets.
 By construction, the functor $\theta$ is given by the formula
 \[
  \theta(V) \simeq \colim_n \Omega^{n-1}(g(V[n-1])) \in \dgLie_\k^\Omega.
 \]
 In particular, we find a functorial morphism in $\dgLie_\k^\Omega$
 \[
  \theta(V) \simeq \colim_n \Omega^{n-1}(g(V[n-1])) \to \colim_n \Omega^{n-1}(h(V[n-1])) \simeq h(V).
 \]
 We have already seen in the proof of \autoref{prop:abelianlie} that the image by $f$ of this morphism is an equivalence. The first result follows by conservativity of $f$.
 
 For the second statement, we simply observe that $\bCEO$ is the left derived functor of the abelianization functor $L \mapsto \quot{L}{[L,L]}$, which is left adjoint to $h$.
\end{proof}

\paragraph{Description of the equivalence $\FMP_\k^\Ab \simeq \dgMod_\k$}

We will give a more explicit description of the equivalence of \autoref{prop:abelianlie}. The definition implies that $\TAb \colon \FMP_\k^\Ab \to^\sim \dgMod_\k$ simply computes the tangent complex (at the only $\k$-point). Let us also describe its inverse $\eAb$.

\begin{lem}\label{lem:abelianization}
 Let $B \in \dgArt_\k$ and $X = \Map(B,-) \in \FMP_\k$.
 We have
 \[
  \T \FAb(X) \simeq \dual{\Aug(B)},
 \]
where $\Aug$ computes the augmentation ideal of a given Artinian, and $\dual{(-)}$ computes the $\k$-linear dual.
\end{lem}

\begin{proof}
Since $B$ is Artinian, we deduce from \autoref{lem:BsimCDB} that $B$ is canonically equivalent to the Chevalley-Eilenberg cohomology of it (shifted) tangent Lie algebra $\T X$.
It follows that the tangent Lie algebra of $\FAb X$ is the dg-module $\bCEO(\T X) \simeq \dual{\Aug (B)}$.
\end{proof}

\begin{prop}\label{prop:calceAb}
 The functor $\eAb$ is equivalent to the functor mapping $V \in \dgMod_\k$ to the abelian formal moduli problem
 \[
  B \mapsto \left(\Aug(B) \otimes_\k V\right)^{\leq 0},
 \]
 where $\Aug$ computes the augmentation ideal of a given Artinian cdga and $(-)^{\leq 0}$ truncates the given complex (and considers it as a connective spectrum through the Dold-Kan equivalence).
\end{prop}

\begin{proof}
 The equivalence $\dgLie^\Omega_\k \simeq \FMP_\k$ is constructed by identifying $\dgArt_\k$ with a full subcategory of $\dgLie^\Omega_\k$ of so-called good (shifted) dg-Lie algebras, that generates $\dgLie_\k^\Omega$ in a certain way. This identification is given by the Chevalley-Eilenberg functor $\CEO$.
 In particular, given $V \in \dgMod_\k$, and $B \in \dgArt_\k$, we have
 \begin{multline*}
  \eAb(V)(B) \simeq \Map_{\FMP_\k^\Ab}\left(\FAb(\Map(B,-)), \eAb(V)\right) \\ \simeq \Map_{\dgMod_\k}\left( \dual{\Aug(B)}, V\right) \simeq \left(\Aug(B) \otimes_\k V\right)^{\leq 0}.
 \end{multline*}

\end{proof}

\paragraph{Diagrammatic summary}

It follows from \autoref{prop:abelianlie} that an abelian formal moduli problem is determined by its tangent complex (without any additional structure).

We denote by $\ellAb$ the composite functor $\ellAb := \TAb \circ \LAb$.
We get the commutative diagrams of right adjoints and of left adjoints
\[
\begin{tikzcd}[column sep=1.6em]
\dgLie_\k^{\Omega}
  \ar[leftarrow]{d}{\theta}
  \ar[yshift=-0.12em]{r}[swap]{\e}
  \ar[yshift=0.33em, leftarrow]{r}{\T}
  \ar[phantom]{r}{\sim}
& \FMP_\k
  \ar[leftarrow]{d}{\GAb}
  \ar[hook]{r}[swap]{\incl} 
& \PFMP_\k
\\
\dgMod_\k
  \ar[yshift=-0.12em]{r}[swap]{\eAb}
  \ar[yshift=0.33em, leftarrow]{r}{\TAb}
  \ar[phantom]{r}{\sim}
& \FMP_\k^\mathrm{Ab}
  \ar[hook]{r}[swap]{\iAb}
& \PFMP^\mathrm{Ab}_\k
  \ar{u}[swap]{\infloop}
\end{tikzcd}
\hspace{0.38em}
\begin{tikzcd}[column sep=1.6em]
\dgLie_\k^{\Omega}
  \ar{d}[swap]{\bCE^\Omega_\bullet}
  \ar[yshift=-0.12em]{r}[swap]{\e}
  \ar[yshift=0.33em, leftarrow]{r}{\T}
  \ar[phantom]{r}{\sim}
& \FMP_\k
  \ar{d}[swap]{\FAb}
& \PFMP_\k
  \ar{l}{\form}
\\
\dgMod_\k
  \ar[yshift=-0.12em]{r}[swap]{\eAb}
  \ar[yshift=0.33em, leftarrow]{r}{\TAb}
  \ar[phantom]{r}{\sim}
& \FMP_\k^\Ab
& \PFMP^\Ab_\k.
  \ar{l}{\LAb}
  \ar[leftarrow]{u}{\infsusp}
  \ar[bend left=20]{ll}[swap]{\ellAb}
\end{tikzcd}
\]

\subsection{\texorpdfstring{$\Q$}{ℚ}-linear moduli problems}\label{subsec:FMPQ}

\paragraph{Definitions}
\begin{df}
 A $\Q$-linear (pre-)FMP is a (pre-)FMP with values in the $\infty$-category $\Cc = \dgMod_\Q^{\leq 0}$.
 We shorten the notations by setting
 \[
 \FMP_\k^\Q := \FMP_\k^\Cc, \hspace{5mm} \PFMP_\k^\Q := \PFMP_\k^\Cc, \hspace{5mm}   \iQ := \iC \hspace{5mm}\text{and}\hspace{5mm} \LQ := \LC.
 \]
 We denote by $\jQ$ the forgetful functor $\PFMP_\k^\Q \to \PFMP^\Ab_\k$ and by $\eQ$ the functor $\dgMod_\k \to \FMP_\k^\Q$ given by the formula
 \[
  \eQ(V) \colon B \mapsto \left(\Aug(B) \otimes_\k V\right)^{\leq 0} \in \dgMod^{\leq 0}_\Q.
 \]
 Recall that $\eAb$ denotes the inverse of the equivalence $\T \colon \FMP^\Ab_\k \to \dgMod_\k$.
\end{df}

\begin{prop}\label{prop:fmpQ}
The following assertions hold.
\begin{assertions}
 \item\label{ass:jFMP} The functor $\jQ \colon \PFMP^\Q_\k \to \PFMP_\k^\Ab$ preserves formal moduli problems and is fully faithful.
 \item\label{ass:iAbfactors} The functor $\eAb$ factors as
 \[
  \eAb \simeq \jQ \circ \eQ \colon \dgMod_\k \to^{\eQ} \FMP_\k^\Q \to^\jQ \FMP_\k^\Ab.
 \]
 \item\label{ass:iAbiQ}
 The functors $\eQ \colon \dgMod_\k \to \FMP^\Q_\k$ and $\operatorname j_\Q \colon \FMP^\Q_\k \to \FMP^\Ab_\k$ are equivalences.
\end{assertions}
As a consequence, we have the following factorization of the adjunction $\LAb \dashv \iAb$ :
\[\label{eq:eAb}
\begin{tikzcd}
\dgMod_\k \simeq \FMP^\Ab_\k \simeq \FMP_\k^\Q
  \arrow[rr, "\iAb", rounded corners, to path={ 
    -- (\tikztostart.south)
    |- ([yshift=-2ex, xshift=-1ex]\tikztotarget.south) [swap,near end]\tikztonodes
    -- ([yshift=0.9ex, xshift=-1ex]\tikztotarget.south)}]
  \ar[shift right, hook']{r}[swap]{\iQ} 
  \ar[shift left, leftarrow]{r}{\LQ}
& \PFMP_\k^\Q 
  \ar[shift right, hook']{r}[swap]{\operatorname j_\Q}
  \ar[shift left, leftarrow]{r}{-\wedge \Q}
& \PFMP_\k^\Ab. 
  \arrow[ll, "\LAb", rounded corners, to path={ 
    ([xshift=-1ex, yshift=-0.7ex]\tikztostart.north)
    |- ([yshift=2ex]\tikztotarget.north) [swap,near end]\tikztonodes
    -- ([yshift=-0.7ex]\tikztotarget.north)}]
\end{tikzcd}
\]
\end{prop}

\begin{proof}
The functor $\jQ$ is given by post-composing with the limit preserving functor $\dgMod_\Q^{\leq 0} \to \Sp_{\geq 0}$. In particular, it preserves the \ref{cond:schlessinger}. Recall that $\Q$ is idempotent in spectra: $\Q \wedge \Q \simeq \Q$. It follows that $\jQ$ is fully faithful. This proves \ref{ass:jFMP}.

\autoref{prop:calceAb} implies \ref{ass:iAbfactors}.
All is left is \ref{ass:iAbiQ}, which follows from \ref{ass:jFMP} and \ref{ass:iAbfactors}.
\end{proof}

\begin{df}
We denote by $\TQ \colon \FMP_\k^\Q \to \dgMod_\k$ the inverse of $\eQ$.
 We denote by $\ellQ$ the composite $\TQ \circ \LQ$.
\end{df}

\begin{rmq}\label{rmq:lbar}
 Since $\k$ is both initial and final in $\dgArt_\k$, any $F \in \PFMP_\k^\Q$ (resp. in $\PFMP_\k^\Ab$) splits as $F \simeq \widebar F \oplus F(\k)$, where $\widebar F$ is pointed (i.e. satisfies the \ref{cond:S1}) and $F(\k)$ is the constant functor.
 Since the inclusion $\iQ$ (resp. $\iAb$) factors through the category of pointed functors, its left adjoint $\LQ$ (resp. $\LAb$) can be decomposed into two functors. The first associates $\widebar F$ to $F$, while the second forces \ref{cond:S2}.
 In particular, we have
 \begin{align*}
  \LQ(F) \simeq \LQ(\widebar F) \text{ }&\text{and } \ellQ(F) \simeq \ellQ(\widebar F) \\ \text{\Big(resp. } \LAb(F) \simeq \LAb(\widebar F) \text{ }&\text{and } \ellAb(F) \simeq \ellAb(\widebar F) \text{\Big)}.
 \end{align*}
\end{rmq}

\paragraph{Generators}
In this paragraph, we will identify families of generators of the category $\PFMP_\k^\Q$.
\begin{df}
For $B \in \dgArt_\k$, we denote by $\SQ(B) \in \PFMP_\k^\Q$ (resp. $\SQb(B)$) the functor
 \[
  \SQ(B) := \mathrm{C}_\bullet \left( \Map_{\dgArt_\k}(B,-), \Q \right) \hspace{1em} \left(\text{resp. }
  \SQb(B) := \widebar{\mathrm{C}}_\bullet \left( \Map_{\dgArt_\k}(B,-), \Q \right) \right),
 \]
 where $\mathrm{C}_\bullet(-,\Q)$ (resp. $\widebar{\mathrm{C}}_\bullet(-,\Q)$) computes the (reduced) rational homology of a given simplicial set.
\end{df}

\begin{lem}\label{lem:generateurs}
 The category $\PFMP_\k^\Q$ is generated under colimits by functors of the form $\SQ(B)$.
 The full subcategory of $\PFMP_\k^\Q$ spanned by pointed functors (i.e. satisfying the \ref{cond:S1}) is generated under colimits by functors of the form $\SQb(B)$.
\end{lem}

\begin{proof}
 Note that $\PFMP_\k^\Q$ is equivalent to the $\infty$-category of $\Q$-linear objects in $\PFMP_\k$. As a category of presheaves, $\PFMP_\k$ is generated under colimits by the representable  functors $\Map_{\dgArt_\k}(B,-)$ (for $B \in \dgArt_\k$).
 It follows that $\PFMP_\k^\Q$ is generated under colimits by the free $\Q$-linear presheaves generated by those $\Map_{\dgArt_\k}(B,-)$, i.e. by the $\SQ(B)$'s.
 The second statement follows.
\end{proof}

We now compute explicitly the formal moduli problem associated to such a generator.
\begin{lem}\label{lem:tgtSQ}
Let $B \in \dgArt_\k$.
There are functorial equivalences $\ellQ(\SQ(B)) \simeq \ellQ(\SQb(B)) \simeq \dual{\Aug(B)}$, where $\dual{(-)}$ computes the $\k$-linear dual and $\Aug$ the augmentation.
\end{lem}
\begin{proof}
From \autoref{rmq:lbar}, we have $\ellQ(\SQb(B)) \simeq \ellQ(\SQ(B))$.
The result then follows from \autoref{lem:abelianization} in conjunction with the factorization from \autoref{prop:fmpQ}.
\end{proof}

\paragraph{Monoidality}

We will now consider monoidal structures on the adjunction
\[
 \ellQ \colon \PFMP_\k^\Q \rightleftarrows \FMP_\k^\Q \simeq \dgMod_\k \noloc \eQ.
\]
We first observe that both sides admit a natural tensor structure: $\otimes_\k$ on the RHS, and the point-wise application of $\otimes_\Q$ on the LHS.

\begin{lem}
 The functor $\eQ$ is non-unitally lax symmetric monoidal.
\end{lem}

\begin{proof}
 We consider the constant moduli problem functor $\dgMod_\k \to \PFMP_\k^\Q$ mapping $V$ to the constant functor $\underline V$. It is right adjoint to the symmetric monoidal functor $F \mapsto F(\k) \otimes_\Q \k$ and therefore inherits a lax monoidal structure.
 
 Let $\Ib$ be the functor $\Ib \colon \dgArt_\k \to \dgMod_\Q$ mapping an Artinian $B$ to its augmentation ideal $\Aug(B)$. It is by construction an ideal in the commutative algebra object $\I \colon B \mapsto B$ and therefore inherits a non-unital commutative algebra structure.
 
 In particular, the functor $\eQ \colon V \mapsto \tau^{\leq 0}\left(\underline V \otimes_{\underline \k} \Ib\right)$ is non-unitally lax symmetric monoidal.
\end{proof}

As a direct consequence of this lemma, we get that the functor $\ellQ$ is non-unitally colax symmetric monoidal.

\begin{prop}\label{prop:monoidal}
The functor $\ellQ$ is non-unitally symmetric monoidal once restricted to the full subcategory of pointed functors.
\end{prop}

\begin{proof}
 We are to prove that for any pair $F, G \in \PFMP_\k^\Q$ of pointed functors (i.e. $F(\k) \simeq G(\k) \simeq 0$), the natural morphism
 \[
  \gamma_{F,G} \colon \ellQ \left(F \otimes_\Q G \right) \to \ellQ(F) \otimes_\k \ellQ(G)
 \]
 is an equivalence. Fixing $F$, we denote by $\Dd_F \subset \PFMP_\k^\Q$ the full subcategory spanned by the $G$'s such that $\gamma_{F,G}$ is an equivalence.
 Since the tensor products $\otimes_\k$ and $\otimes_\Q$ preserve colimits in each variable, and since $\ellQ$ preserves colimits, the category $\Dd_F$ is stable under colimits.
 Using \autoref{lem:generateurs}, we can therefore reduce the question to the case where $G$ (and by symmetry, also $F$) is of the form $\SQb(B)$.
 Let $B_1$ and $B_2$ be Artinian cdga's over $\k$, and assume that $F = \SQb(B_1)$ and $G = \SQb(B_2)$.
 
 The reduced Künneth formula provides a (functorial) equivalence 
 \begin{multline*}
  \SQb\left(B_1 \otimes_\k B_2\right) \simeq \widebar{\mathrm{C}}_\bullet \left( \Map_{\dgArt_\k}(B_1,-) \times \Map_{\dgArt_\k}(B_2,-), \Q \right) \\ \simeq \SQb(B_1) \otimes_\Q \SQb(B_2) \oplus \SQb(B_1) \oplus \SQb(B_2).
 \end{multline*}
 Applying $\ellQ$ on that equivalence, we find using \autoref{lem:tgtSQ}
 \[
  \dual{\Aug\left(B_1 \otimes_\k B_2\right)} \simeq \ellQ\left(\SQb(B_1) \otimes_\Q \SQb(B_2)  \right) \oplus \dual{\Aug(B_1)} \oplus \dual{\Aug(B_2)}.
 \]
 Since $B_1$ and $B_2$ are perfect as $\k$-dg-modules, the LHS identifies with
 \[
  \dual{\Aug(B_1)} \otimes_\k \dual{\Aug(B_2)} \oplus \dual{\Aug(B_1)} \oplus \dual{\Aug(B_2)}.
 \]
The map $\gamma_{\SQb(B_1), \SQb(B_2)}$ is thus a retract of the equivalence $\ellQ(\SQb(B_1 \otimes_\k B_2)) \simeq \dual{\Aug(B_1 \otimes_\k B_2)}$ and is therefore itself an equivalence.
\end{proof}
\begin{lem}\label{lem:Cactionell}
Let $C \in \cdga_\Q$ and $V \in \dgMod_C^{\leq 0}$. Let $F \colon \dgArt_\k \to \dgMod_C^{\leq 0}$ be a pre-FMP. Tensoring point-wise by $V$ defines a new pre-FMP $F \otimes_C V$. The canonical morphism
 \[
  \ellQ\left(F \otimes_C V\right) \to (\ellQ F) \otimes_C V
 \]
is an equivalence in $\dgMod_\k$.
\end{lem}

\begin{proof}
 Since the involved functors preserves all colimits, we can reduce to the generating case $V = C$, which is trivial.
\end{proof}

\begin{cor}\label{cor:tensorellzero}
 Let $F, G \in \PFMP_\k^\Q$. If $F$ is pointed (i.e. $F(\k) \simeq 0)$) and if $\ellQ(F) \simeq 0$, then $\ellQ(F \otimes_\Q G) \simeq 0$.
\end{cor}
\begin{proof}
We split $G$ into the direct sum $\bar G \oplus G(\k)$ where $\bar G$ is pointed and $G(\k)$ is a constant functor. We can therefore assume that $G$ is either pointed or constant.
The first case follows from \autoref{prop:monoidal}, while the second case follows from \autoref{lem:Cactionell} (for $C = \Q$).
\end{proof}

\section{Excision}
In this section, we fix $\Ac$ an algebra object in $\PFMP_\k^\Q$.
We also denote by $\bAc$ the associated pointed functor
\[
 \bAc \colon B \mapsto \hofib( \Ac(B) \to \Ac(\k) ).
\]
Note that $\bAc$ inherits a non-unital algebra structure.
Finally, we denote by $A$ the (non-unital) algebra $\ellQ(\Ac) \simeq \ellQ(\bAc)$ in $\dgMod_\k$.

\subsection{Main theorem}

\begin{thm}\label{thm:excision}
Consider the canonical morphisms
 \[
 \begin{tikzcd}[row sep=0]
  \HH_\bullet^\Q(\bAc) \ar{r}{\alpha_\HH} & \HH_\bullet^\Q(\Ac) & \ellQ(\HH_\bullet^\Q(\Ac)) \ar{r}{\beta_\HH} & \HH_\bullet^\k(A), \\
  \HC_\bullet^\Q(\bAc) \ar{r}[swap]{\alpha_\HC} & \HC_\bullet^\Q(\Ac) & \ellQ(\HC_\bullet^\Q(\Ac)) \ar{r}[swap]{\beta_\HC} & \HC_\bullet^\k(A).
 \end{tikzcd}
 \]
The following holds:
\begin{assertions}
 \item\label{ass:betagamma} The morphisms $\beta_\HH \circ \ellQ(\alpha_\HH)$ and $\beta_\HC \circ \ellQ(\alpha_\HC)$ are equivalences.
 \item\label{ass:alpha} If $A$ is H-unital, then the morphisms $\ellQ(\alpha_\HH)$ and $\ellQ(\alpha_\HC)$ are equivalences (and therefore so are $\beta_\HH$ and $\beta_\HC$).
\end{assertions}
\end{thm}
\begin{rmq}
 A direct consequence of \ref{ass:alpha} is that the tangent complex of Hochschild or cyclic homology of a given functor $\Ac$ does not depend on $\Ac(\k)$. This is an excision statement similar to \autoref{thm:wodzicki}.
\end{rmq}
\begin{cor}\label{cor:tgtK}
 Let $C \in \dgAlg_\k^{\mathrm{nu}}$ be H-unital. Denote by $\Ac_C$ the functor $B \mapsto (C \otimes_\k B)^{\leq 0}$.
 There is a functorial equivalence
 \[
  \ellAb(\K(\Ac_C)) \simeq \HC_\bullet^\k(C)[1].
 \]
\end{cor}

\begin{proof}
  We have $\bAc_C = \eQ(C)$ and therefore $\ellQ(\Ac_C) \simeq \ellQ(\bAc_C) \simeq C$.
  We find
 \begin{align*}
  \ellAb(\K(\Ac_C)) &\simeq \ellAb(\bK(\Ac_C)) & &\text{by \autoref{rmq:lbar}}
  \\ &\simeq \ellQ( \bK(\Ac_C) \wedge \Q)  & &\text{by \autoref{prop:fmpQ}}
  \\ &\simeq \ellQ(\bHC_\bullet^\Q(\Ac_C)[1])  & &\text{by \autoref{cor:goodwillieHunital}}
  \\ &\simeq \ellQ(\HC_\bullet^\Q(\Ac_C))[1]   & &\text{by \autoref{rmq:lbar}}
  \\ &\simeq \HC_\bullet^\k(C)[1]  & &\text{by \autoref{thm:excision}}.
 \end{align*}
\end{proof}

\begin{rmq}\label{rmk:ellKnc}
Using the notations of \autoref{df:FMPC} for $\Cc = \Sp$ the category of spectra and the proof of \autoref{prop:abelianlie}, we get that $\FMP^\Sp_\k$ ($\simeq \Cc_3$) is equivalent to $\dgMod_\k$.
 We get an adjunction
\[
 \ell^\Sp \colon \PFMP_\k^\Sp \rightleftarrows \FMP_\k^\Sp \simeq \dgMod_\k \noloc \e_\Sp.
\]
It follows from \autoref{rmk:relKnc} and our main theorem that for any algebra object $\Ac$ in $\PFMP_\k^\Q$ such that $A := \ellQ(\Ac)$ is H-unital, we have
\[
 \ell^\Sp (\K^\mathrm{nc}(\Ac)) \simeq \ellAb (\K(\Ac)) \simeq \HC^\k_\bullet(A)[1].
\]
\end{rmq}

\begin{rmq}
 The equivalence of \autoref{cor:tgtK} is defined through the relative Chern character. We know from \cite{cathelineau:lambda} and \cite{chw:lambda} that the Chern character is compatible with the $\lambda$-operations on both sides. It follows that the equivalence of \autoref{cor:tgtK} is also compatible with the $\lambda$-operations. 
\end{rmq}

\begin{rmq}[Goodwillie derivative]\label{rmk:goodwilliederivative}
 Denote by $f$ the functor $\Perf_\k^{\leq 0} \to \dgArt_\k$ mapping a connective perfect $\k$-complex $M$ to the split square zero extension $\k \oplus M$. Restricting along $f$ defines a functor from $\PFMP_\k^{\Sp_{\geq 0}}$ to the category of functors $F \colon \Perf_\k^{\leq 0} \to \Sp_{\geq 0}$. Such functors $F$ satisfying some Schlessinger-like condition form a category equivalent to that of $\k$-complexes. We find a commutative diagram
 \[
 \begin{tikzcd}
  \Fct(\Perf_\k^{\leq 0}, \Sp_{\geq 0}) && \ar{ll}[swap]{- \circ f} \PFMP_\k^{\Ab} \\
& \dgMod_\k \ar{ur}[swap]{\eAb} \ar{ul}{\phi}
 \end{tikzcd}
 \]
 where $\phi$ (as well as $\eAb$) is fully faithful. Now, the left adjoints $\ellAb$ and (say) $\psi$ of $\eAb$ and $\phi$ respectively, do not commute with $- \circ f$.
 The functor $\psi$ can actually be interpreted in terms of Goodwillie derivative. For instance, it can be proved to map $\K(\Ac_\k) \circ f$ to $\k \in \dgMod_\k$. Note that $\k[1]$, seen as a $\k \otimes \k$-module, corepresents the Goodwillie derivative of the $\K$-theory of $\k$:
 \[
 \partial \K \simeq \HH_\bullet(\k, -)[1] \colon \dgMod_{\k \otimes \k} \to \Sp.
 \]
 More generally, one can prove $\psi(\K(\Ac_C) \circ f) \simeq \HH_\bullet^\k(C)[1]$. Moreover, the Beck-Chevalley transformation $\psi \circ (- \circ f) \to \ellAb$ is identified with the usual morphism $\HH_\bullet^\k(C)[1] \to \HC_\bullet^\k(C)[1]$.
 
 We will not use the above remark in what follows, so we will not provide a proof.
 Note however that understanding the relationship between $\psi$ and $\ellAb$ (and a putative circle action) may give us a less computational (and more conceptual) proof of \autoref{cor:tgtK}, based directly on Goodwillie calculus.
\end{rmq}

\begin{lem}\label{lem:ellBH}
 Let $\Mc \in \PFMP_\k^\Q$ be an $\bAc$-bimodule. Assume that $\Mc(\k) \simeq 0$. We set $M := \ellQ(\Mc)$ as an $A$-bimodule.
 The canonical morphisms 
 \begin{align*}
  \beta\alpha_\Bc &\colon \ellQ\left(\Bc_\bullet^\Q(\bAc, \Mc) \right) \to \Bc_\bullet^\k(A,M)\\
  \beta\alpha_\Hc &\colon \ellQ\left(\Hc_\bullet^\Q(\bAc, \Mc) \right) \to \Hc_\bullet^\k(A,M)
 \end{align*}
 are equivalences.
\end{lem}

\begin{proof}
 The augmented Bar complex $\Bc_\bullet^\Q(\bAc,\Mc)$ identifies as the homotopy cofiber of the augmentation $\widebar{\Bc}_\Q(\bAc,\Mc) \to \Mc$, where $\widebar{\Bc}_\Q(-,-)$ denotes the reduced Bar complex. The latter is obtained as a homotopy colimit of the semi-simplicial Bar construction.
 Since $\ellQ$ preserves colimits, we find using \autoref{prop:monoidal}
 \[
  \ellQ\left(\widebar{\Bc}_\Q(\bAc,\Mc)\right) \simeq \ellQ\left(\colim_{[n] \in \Delta} \Mc \otimes_\Q \bAc^{\otimes_\Q n}\right) \simeq \colim_{[n] \in \Delta} M \otimes_\k A^{\otimes_\k n} \simeq \widebar{\Bc}_\k(A,M).
 \]
 Taking the homotopy cofiber of the augmentation on both sides, we find the first claimed equivalence.
 Similarly, the functor $\Hc_\bullet^\Q(\bAc, \Mc)$ is again the homotopy colimit of a standard semi-simplicial diagram.
\end{proof}

\begin{proof}[of \autoref{thm:excision} (a)]
 The functor $\HH_\bullet^\Q$ is the homotopy cofiber of the natural transformation $1-t \colon \Bc_\bullet^\Q \to \Hc_\bullet^\Q$. In particular, the morphism $\beta_{\HH} \circ\ellQ(\alpha_{\HH})$ is a equivalence because of \autoref{lem:ellBH} (for $\Mc = \bAc$) and the fact that $\ellQ$ preserves cofiber sequences.
 In the case of $\HC$, we use \autoref{rmq:HChS} and \autoref{lem:Cactionell}:
 \begin{multline*}
  \ellQ\left(\HC_\bullet^\Q(\bAc)\right) \simeq \ellQ\left(\HH_\bullet^\Q(\bAc) \otimes_{\Q[\epsilon]} \Q \right) \simeq \ellQ\left(\HH_\bullet^\Q(\bAc)\right) \otimes_{\Q[\epsilon]} \Q
  \\\simeq \HH_\bullet^\k(A) \otimes_{\Q[\epsilon]} \Q
  \simeq \HC_\bullet^\k(A),
 \end{multline*}
 where $\Q[\epsilon] := \homol_\bullet(S^1, \Q) \in \cdga_\Q$.
\end{proof}

\subsection{Proof of \ref{ass:alpha}}
The proof of \ref{ass:alpha} in \autoref{thm:excision} is more evolved and relies on the ideas behind Wodzicki's \autoref{thm:wodzicki}. We first reduce the study of Hochschild and cyclic homology to that of the complexes $\Hc$ and $\Bc$ from \autoref{subsec:cyclic}.
We will use the following terminology:
 \begin{df}
  A morphism $f \colon F \to G \in \PFMP_\k^\Q$ is an  $\ellQ$-equivalence if $\ellQ(f)$ is an equivalence.
 \end{df}
 
Denote by $\alpha_\Hc$ and $\alpha_\Bc$ the canonical morphisms
\[
\begin{tikzcd}[row sep=0]
  \Hc_\bullet^\Q(\bAc) \ar{r}{\alpha_\Hc} & \Hc_\bullet^\Q(\Ac), \\
  \Bc_\bullet^\Q(\bAc) \ar{r}[swap]{\alpha_\Bc} & \Bc_\bullet^\Q(\Ac).
 \end{tikzcd}
 \]
 As in the previous section, we can easily reduce the proof of \autoref{thm:excision}, \ref{ass:alpha} to proving that both $\alpha_\Hc$ and $\alpha_\Bc$ are $\ellQ$-equivalences.

 \begin{lem}
  If both $\alpha_\Hc$ and $\alpha_\Bc$ are $\ellQ$-equivalences, then so are $\alpha_\HH$ and $\alpha_\HC$.
 \end{lem}
 
 We will now focus on $\alpha_\Hc$ and $\alpha_\Bc$ using techniques from \autoref{subsec:wodzicki}.
 
\begin{prop}\label{prop:alphaelleq}
 The canonical morphisms
 \[
  \alpha_\Bc \colon \Bc_\bullet^\Q(\bAc) \to \Bc_\bullet^\Q(\Ac) \hspace{2em} \text{and} \hspace{2em} \alpha_\Hc \colon \Hc_\bullet^\Q(\bAc) \to \Hc_\bullet^\Q(\Ac)
 \]
 are $\ellQ$-equivalences.
\end{prop}

 \begin{lem}\label{lem:gammaeqs}
  Let $\Mc \in \PFMP_\k^\Q$ be an $\Ac$-bimodule.
  Assume that $\Mc(\k) \simeq 0$ and that $M:= \ellQ(\Mc)$ is H-unital as an $A = \ellQ(\bAc)$-bimodule, then
  \[
   \gamma_\Hc \colon \Hc_\bullet^\Q(\bAc,\Mc) \to \Hc_\bullet^\Q(\Ac, \Mc) \hspace{2em} \text{and} \hspace{2em} 
   \gamma_\Bc\colon\Bc_\bullet^\Q(\bAc,\Mc) \to \Bc_\bullet^\Q(\Ac, \Mc)
  \]
  are $\ellQ$-equivalences.
 \end{lem}

 \begin{proof}
  We focus on $\gamma_\Hc$, the case of $\gamma_\Bc$ being identical.
  Denote by $f$ the canonical natural transformation $\Ac \to \Ac(\k)$ (where on the RHS is the constant functor). We get $\bAc \simeq \hofib(f)$.
  Recall from \autoref{subsec:wodzicki} \autoref{par:filtrationF} the filtration
  \begin{multline*}
   \Hc_\bullet^\Q(\bAc,\Mc) \simeq \Fc_{\Hc_\bullet^\Q}^0(f,\Mc) \to \cdots \to \Fc_{\Hc_\bullet^\Q}^n(f,\Mc) \to \cdots
   \\
   \text{with  }\colim_n \Fc_{\Hc_\bullet^\Q}^n(f,\Mc) \simeq \Hc_\bullet^\Q(\Ac, \Mc).
  \end{multline*}
  Applying $\ellQ$, we find
  \[
   \ellQ\left(\Hc_\bullet^\Q(\bAc,\Mc)\right) \simeq \ellQ\left(\Fc_{\Hc_\bullet^\Q}^0(f,\Mc)\right) \to \cdots \to \ellQ\left(\Fc_{\Hc_\bullet^\Q}^n(f,\Mc)\right) \to \cdots
  \]
  Since $\ellQ$ is a left adjoint and thus preserves colimits, we have 
  \[
   \colim_n \ellQ\left(\Fc_{\Hc_\bullet^\Q}^n(f,\Mc)\right) \simeq \ellQ\left(\colim_n \Fc_{\Hc_\bullet^\Q}^n(f,\Mc)\right) \simeq \ellQ\left(\Hc_\bullet^\Q(\Ac, \Mc)\right).
  \]
  It therefore suffices to prove that for any $n \geq 0$, the morphism $\Fc_{\Hc_\bullet^\Q}^n(f,\Mc) \to \Fc_{\Hc_\bullet^\Q}^{n+1}(f,\Mc)$ is an $\ellQ$-equivalence. Denote by $F^n$ the complex $\ellQ(\Fc_{\Hc_\bullet^\Q}^n(f,\Mc))$.
  Since $\dgMod_\k$, the codomain of $\ellQ$, is a stable $\infty$-category, it is enough to check that
  \[
   \quot{F^{n+1}}{F^n} \simeq \ellQ\left( \quot{\Fc_{\Hc_\bullet^\Q}^{n+1}(f,\Mc)}{\Fc_{\Hc_\bullet^\Q}^n(f,\Mc)} \right)
  \]
  is contractible.
  Using \autoref{lem:filtrationF}, we get 
  \[
   \quot{\Fc_{\Hc_\bullet^\Q}^{n+1}(f,\Mc)}{\Fc_{\Hc_\bullet^\Q}^n(f,\Mc)} \simeq \Ac^{\otimes n} \otimes \Ac(\k) \otimes \Bc_\bullet^\Q(\bAc, \Mc)[n+1].
  \]
  From \autoref{lem:ellBH}, we have $\ellQ\left(\Bc_\bullet^\Q(\bAc, \Mc) \right) \simeq \Bc_\bullet^\k(A,M) \simeq 0$ (using the assumption that $M$ is H-unital over $A$).
  Since the functor $\Bc_\bullet^\Q(\bAc, \Mc)$ is pointed, we get that $\quot{F^{n+1}}{F^n} \simeq 0$ using \autoref{cor:tensorellzero}. We conclude that $\gamma_\Hc$ is an $\ellQ$-equivalence.
 \end{proof}

\begin{lem}\label{lem:deltaeqs}
 If $A$ is H-unital, the four canonical morphisms
 \begin{align*}
  \Hc_\bullet^\Q(\Ac,\Ac(\k)) &\lra^{\delta_\Hc} \Hc_\bullet^\Q(\Ac(\k),\Ac(\k)) \lra 0\\
  \Bc_\bullet^\Q(\Ac,\Ac(\k)) &\lra^{\delta_\Bc} \Bc_\bullet^\Q(\Ac(\k),\Ac(\k)) \lra 0
 \end{align*}
 are $\ellQ$-equivalences.
\end{lem}

\begin{proof}
 The functors $\Hc_\bullet^\Q(\Ac(\k), \Ac(\k))$ and $\Bc_\bullet^\Q(\Ac(\k), \Ac(\k))$ are constant. Their images by $\ellQ$ thus vanish, by \autoref{rmq:lbar}. It follows that their projections to $0$ are $\ellQ$-equivalences.
 We now focus on the maps $\delta_\Hc$ and $\delta_\Bc$. Let $f \colon \Ac \to \Ac(\k)$ be the augmentation.
 Recall from \autoref{subsec:wodzicki} \autoref{par:quotfiltration} the filtrations by quotients:
 \begin{align*}
   \Bc_\bullet^\Q(\Ac,\Ac(\k)) &\simeq \Qc_{\Bc_\bullet^\Q}^0(f) \to \cdots \to \Qc_{\Bc_\bullet^\Q}^n(f) \to \cdots \to \colim_n \Qc_{\Bc_\bullet^\Q}^n(f) \simeq \Bc_\bullet^\Q(\Ac(\k),\Ac(\k))\\
   \Hc_\bullet^\Q(\Ac,\Ac(\k)) &\simeq \Qc_{\Hc_\bullet^\Q}^0(f) \to \cdots \to \Qc_{\Hc_\bullet^\Q}^n(f) \to \cdots \to \colim_n \Qc_{\Hc_\bullet^\Q}^n(f) \simeq \Hc_\bullet^\Q(\Ac(\k),\Ac(\k)).
 \end{align*}
 Since $\ellQ$ preserves colimits, it suffices to prove that the transition morphisms $\Qc^n \to \Qc^{n+1}$ are $\ellQ$-equivalences.
 Denote by $\Mc(n)$ the $\Ac$-bimodule $\Ac(\k)^{\otimes_\Q n+1} \otimes_\Q \bAc$ and by $M(n)$ its image by $\ellQ$. We get from \autoref{lem:kernelsfiltration} two fiber and cofiber sequences
 \begin{align*}
  \Bc_\bullet^\Q(\Ac, \Mc(n))[n+1] &\lra \Qc_{\Bc_\bullet^\Q}^n(f) \lra \Qc_{\Bc_\bullet^\Q}^{n+1}(f) \\
  \Hc_\bullet^\Q(\Ac, \Mc(n))[n+1] &\lra \Qc_{\Hc_\bullet^\Q}^n(f) \lra \Qc_{\Hc_\bullet^\Q}^{n+1}(f).
 \end{align*}
 Their image by $\ellQ$ are still cofiber sequences, and it is now enough to prove that both $\Bc_\bullet^\Q(\Ac, \Mc(n))$ and $\Hc_\bullet^\Q(\Ac, \Mc(n))$ are cancelled by $\ellQ$.
 We first observe that $\Mc(n)$ is pointed: $\Mc(n)(\k) \simeq 0$. Moreover, we have by \autoref{lem:Cactionell}
 \[
  M(n) = \ellQ(\Mc(n)) \simeq \Ac(\k)^{\otimes_\Q n+1} \otimes_\Q \ellQ(\bAc) = \Ac(\k)^{\otimes_\Q n+1} \otimes_\Q A.
 \]
 \autoref{rmq:MtimesAHunital} implies that $M(n)$ is H-unitary over $A$. Using \autoref{lem:gammaeqs}, we are reduced to the study of $\Bc_\bullet^\Q(\bAc, \Mc(n))$ and $\Hc_\bullet^\Q(\bAc, \Mc(n))$. By \autoref{lem:ellBH}, we get
 \[
  \ellQ\left( \Bc_\bullet^\Q(\bAc, \Mc(n)) \right) \simeq \Bc_\bullet^\k(A,M(n)) \simeq 0
 \]
 and, since the left action of $A$ on $M(n)$ is trivial:
 \[
  \ellQ\left( \Hc_\bullet^\Q(\bAc, \Mc(n)) \right) \simeq \Hc_\bullet^\k(A,M(n)) \simeq \Bc_\bullet^\k(A,M(n)) \simeq 0.
 \]
\end{proof}

\begin{proof}[of \autoref{prop:alphaelleq}]
We first observe that $\alpha_\Bc$ and $\alpha_\Hc$ factor as
\begin{align*}
 \Bc_\bullet^\Q(\bAc) &\lra^{\gamma_\Bc} \Bc_\bullet^\Q(\Ac,\bAc) \lra^{\eta_\Bc} \Bc_\bullet^\Q(\Ac)\\
 \Hc_\bullet^\Q(\bAc) &\lra^{\gamma_\Hc} \Hc_\bullet^\Q(\Ac,\bAc) \lra^{\eta_\Hc} \Hc_\bullet^\Q(\Ac).
\end{align*}
Since $A$ is H-unital and $\bAc$ is pointed, \autoref{lem:gammaeqs} implies that $\gamma_\Bc$ and $\gamma_\Hc$ are $\ellQ$-equivalences.
The homotopy cofibers of $\eta_\Bc$ and $\eta_\Hc$ are respectively $\Bc_\bullet^\Q(\Ac, \Ac(\k))$ and $\Hc_\bullet^\Q(\Ac, \Ac(\k))$. They are cancelled by $\ellQ$ because of \autoref{lem:deltaeqs}. It follows that $\eta_\Bc$ and $\eta_\Hc$ are $\ellQ$ equivalences, and so are $\alpha_\Bc$ and $\alpha_\Hc$.

This concludes the proof of \autoref{prop:alphaelleq} and therefore the proof of \autoref{thm:excision}.
\end{proof}

\subsection{Non-connective \texorpdfstring{$\K$}{K}-theory and the case of schemes}
We now extend our main result (\autoref{cor:tgtK}) to the case of quasi-compact quasi-separated (and possibly derived) schemes.

\begin{prop}
 Let $X$ be a quasi-compact quasi-separated (and possibly derived) scheme over $\k$.
 Denote by $\K_X^\mathrm{nc} \colon \dgArt_\k \to \Sp$ the functor mapping an Artinian dg-algebra $B$ to the non-connective $\K$-theory spectrum of the derived scheme $X \otimes_\k B = X \times \Spec(B)$.
 Using the notations of \autoref{rmk:ellKnc}: the Chern character induces an equivalence
 \[
  \ell^\Sp(\K^\mathrm{nc}_X) \simeq \HC^\k_\bullet(X)[1].
 \]
\end{prop}

\begin{proof}
 We write $X \simeq \colim_i \Spec(A_i)$ as a finite colimit of Zariski open affine subschemes. By Zariski descent for $\K$-theory, we have $\K_X^\mathrm{nc} \simeq \lim_i \K^\mathrm{nc}(\Ac_{A_i})$. Since $\ell^\Sp$ is an exact functor between stable $\infty$-categories, it preserves finite limits. We find, using descent for cyclic homology and \autoref{cor:tgtK}:
 \[
  \ell^\Sp(\K_X) \simeq \ell^\Sp(\lim \K^\mathrm{nc}(\Ac_{A_i})) \simeq \lim \ell^\Sp \K^\mathrm{nc}(\Ac_{A_i}) \simeq \lim \HC^\k_\bullet(A_i)[1] \simeq \HC^\k_\bullet(X)[1].
 \]
\end{proof}

\begin{ex}[Relation to the Picard stack]

Let us fix a quasi-compact quasi-separated (derived) scheme $X$ and consider the (abelian) pre-FMP $\PicZ_X$:
\[
 \PicZ_X \colon B \mapsto \{0\} \times_{\PicZ(X)} \PicZ(X \otimes B),
\]
where $\PicZ$ is the graded Picard functor.
It follows from \autoref{ex:artinstacksareFMP} that $\PicZ_X$ satisfies the \ref{cond:schlessinger} and is therefore a formal moduli problem.

Recall that the determinant defines a functorial morphism of spectra from $\K$-theory to the graded Picard group $\det \colon \K^\mathrm{nc} \to \PicZ$.
In particular, we get a morphism of abelian pre-FMP's $\det_X \colon \K^\mathrm{nc}_X \to \PicZ_X$.
Taking the tangent complex yields a morphism
\[
 \mathrm{tr}_X \colon \HC_\bullet^\k(X)[1] \simeq \ell^\Sp(\K^\mathrm{nc}_X) \to \ellAb(\PicZ_X) \simeq \T(\PicZ_X) \simeq \R\Gamma(X, \Oo_X)[1].
\]
The last equivalence is provided by \autoref{ex:tgtcoincides} and the classical computation of the tangent of the Picard stack of $X$ at the trivial bundle.

The canonical morphism $\B\Gm \to \K$ similarly induces a section $s_X$ of $\mathrm{tr}_X$. This identifies $\R\Gamma(X, \Oo_X)$ as a summand of $\HC_\bullet^\k(X)$, which corresponds to the weight $0$ part in the Hodge decomposition of cyclic homology (see \cite{weibel:cyclichodge}).
\end{ex}

\section{Application: the generalized trace map}
\label{sec:tracemap}

Let $A$ be a connective unital dg-algebra over $\k$. Recall from \autoref{subsec:cyclic}, \autoref{par:tracemap}, that the generalized trace map is a (functorial) morphism
\[
 \Tr \colon \CE^\k_\bullet(\gl_\infty(A)) \to \HC_\bullet^\k(A) [1].
\]
A morphism $\CE^\k_\bullet(\gl_\infty(A)) \to \HC_\bullet^\k(A)[1]$ such as $\Tr$ amounts to an $\Lc_\infty$-morphism $\gl_\infty(A) \to \HC_\bullet^\k(A)$, where the RHS is considered with its abelian $\Lc_\infty$-structure. It corresponds to a map $\Tr \colon \gl_\infty(A) \to \HC_\bullet^\k(A)$ in the $\infty$-category $\dgLie_\k$, or equivalently to a map 
\[
 \gl_\infty(A)[1] \to \theta (\HC_\bullet^\k(A)[1])
\]
in the $\infty$-category $\dgLie_\k^\Omega$ of shifted dg-Lie algebras. Recall that $\theta \colon \dgMod_\k \to \dgLie_\k^\Omega$ maps a $\k$-dg-module to the abelian shifted dg-Lie algebra built on that module.

In this section, we will prove that the generalized trace $\Tr$ is tangent (in the sense of formal moduli problems) to the canonical morphism of functors $\BGL \to \K$ mapping a vector bundle to its class in $\K$-theory.
We will see that $\BGL \to \K$ induces a tangent morphism $T \colon \gl_\infty(A) \to \theta(\HC_\bullet^\k(A))$ of dg-Lie algebras over $\k$ and that $T$ is homotopic to the generalized trace $\Tr$. See \autoref{thm:compLQT} below for a precise statement.

\subsection{The tangent Lie algebra of \texorpdfstring{$\BGL$}{BGL}}
Let $\Ac_A \colon \dgArt_\k \to \dgAlg_\Q^{\leq 0}$ denote the functor $B \mapsto A \otimes_\k B$. Denote by $\bAc_A$ its augmentation ideal $\bAc_A(B) \simeq A \otimes_\k \Aug(B)$.

\begin{df} Let $n \in \N \cup \{\infty\}$ and let $\Ac \colon \dgArt_\k \to \dgAlg_\Q^{\leq 0}$ be any functor.
 We denote by $\bBGL_n(\Ac)$ the functor $\dgArt_\k \to \sSets$ mapping $B \in \dgArt_\k$ to
 \[
  \bBGL_n(\Ac)(B) := \hofib(\BGL_n(\Ac(B)) \to \BGL_n(\Ac(\k))).
 \]
\end{df}

\begin{rmq}
 The canonical morphism $\colim_{n \in \N} \bBGL_n(\Ac) \to \bBGL_\infty(\Ac)$ is an equivalence.
\end{rmq}

\begin{ex}
 We assume $\Ac = \Ac_A$.
 Denote by $\exp(\gl_n(\bAc_A))(B)$ the (nilpotent) subgroup of $\GL_n(\Ac_{A}(B))$ of matrices of the form $1 + M$, where $M$ is a matrix with coefficients in $\bAc_A(B) \simeq A \otimes_\k \Aug(B)$.
 We can then identify the homotopy fiber $\bBGL_n(\Ac_A)$ with
 \[
  \bBGL_n(\Ac_A) \simeq \quot{\GL_n(\Ac_A(\k))}{\GL_n(\Ac_A)} \simeq \B \exp(\gl_n(\bAc_A)).
 \]
\end{ex}

For later use, we will work in a slightly bigger generality.
Let $R$ be a (possibly non-unital) discrete $\k$-algebra. We denote by $\g_R$ the functor
\[
 \g_R \colon \dgAlg_\k^{\leq 0,\mathrm{nu}} \to \dgLie_\k^{\leq 0}
\]
mapping a connective dg-algebra $C$ to the Lie algebra underlying the associative algebra $C \otimes_\k R$.
Examples include the case where $R$ is the algebra of $n \times n$-matrices with coefficients in $\k$, where we find $\g_R = \gl_n$.
Obviously, the construction $R \mapsto \g_R$ is functorial.

For an algebra $R$, we also denote by $\bBG_R(\Ac_A) \in \PFMP_\k$ the functor
\[
 \bBG_R(\Ac_A) \colon B \mapsto \B \exp\left(\g_R\left(\bAc_A(B)\right)\right) = \B \exp\left(\g_R\left(A \otimes_\k \Aug(B)\right)\right),
\]
where $\exp\left(\g_R\left(A \otimes_\k \Aug(B)\right)\right)$ is the subgroup of $(A \otimes_\k B \otimes_\k R)^{\times}$ consisting of elements of the form $1 + M$ with $M \in A \otimes_\k \Aug(B) \otimes_\k R$.

\begin{lem}\label{lem:tgtBGL}
There is a functorial (in $R$) equivalence
 \[
  \ell(\bBG_R(\Ac_A)) \simeq \g_R(A)[1] \in \dgLie_\k^\Omega.
 \]
\end{lem}
\begin{proof}
 Denote by $R^+ = \k \ltimes R$ the unital $\k$-algebra obtained by formally adding a unit to $R$.
 Consider the functor $F \colon \dgArt_\k \to \sSets$ mapping $B$ to the maximal $\infty$-groupoid in
 \[
  \Perf_{R^+ \otimes_\k A \otimes_\k B}.
 \]
Deformations of the perfect module $R \otimes_\k A$ are then controlled by the functor $\operatorname{Def}(R \otimes_\k A) \colon B \mapsto F(B) \otimes_{F(\k)} \{R \otimes A\}$. It follows from \cite[Cor. 5.2.15 and Thm. 3.3.1]{lurie:dagx} that the deformations of $R \otimes_\k A$ are controlled by the dg-Lie algebra $\g_R(A) = \End(R \otimes_\k A)$.
 Moreover, we have a natural transformation $\bBG_R(\Ac_A) \to \operatorname{Def}(R \otimes_\k A)$ which induces an equivalence on the loop groups. It is therefore an $\ell$-equivalence and the result follows.
\end{proof}

\begin{cor}
 For any $n \in \N \cup \{\infty\}$, the (shifted) tangent Lie algebra of $\bBGL_n(\Ac_A)$ is $\gl_n(A)[1]$.
\end{cor}

\subsection{The generalized trace}
 Let $\Ac \colon \dgArt_\k \to \dgAlg_\Q^{\leq 0}$ be any functor.
 It comes with a canonical natural transformation $\bBGL_\infty(\Ac) \to \infloop \bK(\Ac)$ mapping a vector bundle to its class.
We can now state the main result of this section:
\begin{thm}\label{thm:compLQT}
 Let $A$ be a connective unital dg-algebra over $\k$. We denote by $\Ac_A \colon \dgArt_\k \to \dgAlg^{\leq 0}_\Q$ the functor $B \mapsto B \otimes_\k A$.
 The natural transformation $\bBGL_\infty(\Ac_A) \to \infloop\bK(\Ac_A)$ induces, by taking the tangent Lie algebras, a morphism
 \[
  T \colon \gl_\infty(A)[1] \simeq \ell(\bBGL_\infty(\Ac_A)) \to \ell(\infloop\bK(\Ac_A)) \to \theta\ellAb(\bK(\Ac_A)) \simeq \theta(\HC_\bullet^\k(A)[1])
 \]
 in $\dgLie_\k^\Omega$.
 The morphism $T$ is homotopic to the generalized trace map $\Tr$.
\end{thm}

\begin{proof}
 Recall that the equivalence $\ellAb(\bK(\Ac_A)) \simeq \HC^\k_\bullet(A)[1]$ is built through the relative Chern character $\bK(\Ac_A) \to \jQ \bHC^\Q_\bullet(\Ac_A)[1]$. In particular, the morphism $T$ is induced by the natural transformation
 \[
  \bch_{\GL} \colon \bBGL_\infty(\Ac_A) \lra \infloop \bK(\Ac_A) \lra \nu(\bHC^\Q_\bullet(\Ac_A)[1])
 \]
where $\nu \simeq \infloop \circ \jQ \colon \PFMP_\k^\Q \to \PFMP_\k$ is the forgetful functor. Denote by $\Q[-]$ the left adjoint to $\nu$ (so that it computes point-wise the rational homology of the given simplicial set).
The natural transformation $\bch_{\GL}$ then factors as
\[
 \bch_{\GL} \colon \bBGL_\infty(\Ac_A) \lra \nu \Q[\bBGL_\infty(\Ac_A)] \lra \nu \Q[X(\Ac_A, \bAc_A)] \lra^{\operatorname{\nu}\bch^\Q} \nu \bHC^\Q_\bullet(\Ac_A)[1].
\]
Recalling the construction of $\bch^\Q$ and using Malcev's theory for $\gl_\infty(\bAc_A)$, we get the commutative diagram
\[
\begin{tikzcd}[column sep=1.3em]
  \bBGL_\infty(\Ac_A) \ar{r} & \nu\Q[\bBGL_\infty(\Ac_A)] \ar{r} \ar{d}[sloped,above]{\simeq} & \nu\Q[X(\Ac_A,\bAc_A)] \ar{r}{\bch^\Q} \ar{d}[sloped,above]{\simeq} & \nu \bHC^\Q_\bullet(\Ac_A)[1] \ar[equal, no head]{d} \\
  & \nu \CE_\bullet^\Q(\gl_\infty(\bAc_A)) \ar{r}{\mathrm{i}} & \displaystyle \nu \colim_{n, \sigma} \CE_\bullet^\Q(\mathrm{t}^\sigma_n(\Ac_A,\bAc_A)) \ar{r}{\Tr} & \nu \bHC_\bullet^\Q(\Ac_A)[1].
\end{tikzcd}
\]
By functoriality of the generalized trace map, we find that the composite map $\Tr \circ \operatorname{i}$ is homotopic to
\[
 \nu \CE_\bullet^\Q(\gl_\infty(\bAc_A)) \lra^{\Tr} \HC_\bullet^\Q(\bAc_A)[1] \lra^{\eta_\HC} \bHC_\bullet^\Q(\Ac_A)[1].
\]
All in all, we get that the natural transformation $\bch_{\GL}$ is homotopic to the composite
\[
 \bBGL_\infty(\Ac_A) \lra \nu \CE_\bullet^\Q(\gl_\infty(\bAc)) \lra \HC_\bullet^\Q(\bAc_A)[1] \lra^{\eta_\HC} \bHC_\bullet^\Q(\Ac_A)[1].
\]
Passing to the tangent morphism and using the Beck-Chevalley natural transformation, we find
\begin{multline*}
T \colon \gl_\infty(A)[1] \to \ell \nu \CE_\bullet^\Q(\gl_\infty(\bAc_A)) \to \theta \ellQ \CE_\bullet^\Q(\gl_\infty(\bAc_A)) \to \theta \ellQ \HC_\bullet^\Q(\bAc_A)[1]\\ \to^\sim \theta \ellQ \bHC_\bullet^\Q(\Ac_A)[1] \simeq \theta \HC_\bullet^\k(A)[1].
\end{multline*}

\begin{lem}
Let $R$ be a (possibly non-unital) $\k$-algebra and $A$ be a (possibly non-unital) connective dg-algebra over $\k$.
 The natural morphism
 \[
  \ellQ \CE_\bullet^\Q(\g_R(\bAc_A)) \to \bCE_\bullet^\k(\g_R(A))
 \]
 induced by the lax monoidal structure on $\eQ$ is an equivalence.
\end{lem}
\begin{rmq}\label{rmq:Risk}
 We have $\g_R(A \otimes_\k -) \simeq (R \otimes_\k A) \otimes_\k -$. Up to replacing $A$ with $R \otimes_\k A$, we can assume that $R = \k$. Note that $\g_\k = \gl_1$ computes the underlying dg-Lie algebra of a given dg-algebra.
\end{rmq}

\begin{proof}
By the above remark, we restrict ourselves to the case $R = \k$.
The functor $\bAc_A$ is pointed. In particular, the morphism $\ellQ(\bAc_A^{\otimes_\Q p}) \to A^{\otimes_\k p}$ is an equivalence for any $p$ (see \autoref{prop:monoidal}). Since $\ellQ$ preserves colimits, we find 
\[
 \ellQ(\Sym^p_\Q(\bAc_A^{\otimes p}[1])) \simeq \Sym^p_\k(A[1]).
\]
The Chevalley-Eilenberg complex comes with a natural filtration whose graded parts are symmetric powers, and the result follows (since $\ellQ$ preserve colimits).
\end{proof}

As a consequence, we get a commutative diagram, for any $A \in \dgAlg_\k^{\leq 0, \mathrm{nu}}$
\[
\begin{tikzcd} 
\ellQ(\CE_\bullet^\Q(\gl_\infty(\bAc_A))) \ar{r}{\ellQ(\Tr)} \ar{d}[sloped, above]{\simeq} & \ellQ(\HC_\bullet^\Q(\bAc_A)[1]) \ar{d}[sloped, above]{\simeq}[swap]{\ellQ(\alpha_{\HC}) \circ \beta_{\HC}} \\
\CE_\bullet^\k(\gl_\infty(A)) \ar{r}{\Tr} & \HC_\bullet^\k(A)[1].
\end{tikzcd}
\]
In particular, the proof of \autoref{thm:compLQT} now reduces to the following
\begin{lem}
For any (possibly non-unital) $\k$-algebra $R$ and any $A \in \dgAlg_\k^{\leq 0, \mathrm{nu}}$, the tangent morphism of the composite
\[
 \bBG_R(\Ac_A) \lra \Q[\bBG_R(\Ac_A)] \simeq \CE_\bullet^\Q(\g_R(\bAc_A))
\]
is homotopic (as a morphism $\g_R(A)[1] \to \theta \ellQ(\CE_\bullet^\Q(\g_R(\bAc_A))) \simeq \theta \bCE_\bullet^\k(\g_R(A))$ in $\dgLie_\k^\Omega$) to the unit of the adjunction $\bCEO(-) = \bCE_\bullet^\k(-[-1]) \dashv \theta$.
\end{lem}

\begin{proof}
By \autoref{rmq:Risk}, we may assume $R = \k$.
\renewcommand{\Alg}{\mathrm{Alg}}
We are to study a morphism $ A[1] \to \theta \bCE^\k_\bullet(A)$ in $\dgLie^\Omega_\k$.
By adjunction, it suffices to identify the corresponding morphism $\bCE_\bullet^\k(A) \to \bCE_\bullet^\k(A)$ with the identity. Note that this is the linear tangent of the Malcev quasi-isomorphism $\Q[\bBG_R(\Ac_A)] \simeq \CE_\bullet^\Q(\g_R(\bAc_A))$. It is actually enough to study the functorial morphism
\[
 \xi_A \colon \CE^\k_\bullet(A) \to \CE^\k_\bullet(A)
\]
obtained by adjoining a unit on both sides.

 This construction is functorial in $A \in \dgAlg_\k^{\leq 0, \mathrm{nu}}$ and moreover the functor $A \mapsto \CE_\bullet^\k(A)$ preserves geometric realizations.
 Since every connective dg-algebra can be obtained as the geometric realization of a diagram of discrete algebras, the natural transformation $\xi$ is determined by its values on discrete algebras. We thus restrict to discrete algebras.
 
 Malcev's quasi-isomorphism is compatible with the standard filtrations. It follows that so is $\xi$. The associated graded part of weight $p$ of both the source and target of $\xi_A$ vanishes for $p \leq 0$ and is canonically isomorphic to $\Sym^p(A[1]) = (\Lambda^p L)[p]$ for positive $p$'s. In particular, both source and target of $\xi_A$ live in the heart of Beilinson's t-structure on filtered complexes (see the appendix of \cite{beilinson:perv}). We can thus in with the $1$-category of complexes.
 
 With Malcev's quasi-isomorphism being compatible with the coalgebras structures, the natural transformation $\xi$ also consists of morphisms of coalgebras. Since $\CE^\k_\bullet \colon \mathrm{Lie}_\k \to \mathrm{coAlg}^\mathrm{filt}_\k$ from discrete Lie algebras to filtered coalgebras is fully faithful, it follows that $\xi$ lifts to a self natural transformation $\zeta \colon F \Rightarrow F$ of the functor $F \colon \mathrm{Alg}_\k^{\mathrm{nu}} \to \mathrm{Lie}_\k$ mapping an algebra to its underlying Lie algebra.

By forgetting down to the category of small sets, we obtain a self natural transformation (still denoted $\zeta$) of the forgetful functor $G \colon \Alg_\k \to \mathrm{Sets}$. The functor $G$ is representable by the non-unital algebra $\k[\mathrm{t}]^{\geq 1}$ of polynomials $P$ over $\k$ such that $P(0) = 0$. By the Yoneda lemma, the natural transformation $\zeta$ is determined by $P \in \k[\mathrm{t}]^{\geq 1}$.
The fact that $\zeta$ is $\k$-linear implies that $P$ is of degree $1$, and the fact that it preserves the Lie brackets implies that $P = a\mathrm{t}$ with $a^2 = a$.
In particular, the natural transformation $\zeta$ (and therefore also $\xi$) is either the identity or the $0$ transformation.
Since $\xi$ is the image by $\ellQ$ of the Malcev equivalence $\Q[\bBG_\k(\bAc_-)] \simeq \CE_\bullet^\Q(\g_\k(\bAc_-))$, it certainly cannot vanish.
\end{proof}

This concludes the proof of \autoref{thm:compLQT}.
\end{proof}

\end{document}